\documentclass[12pt]{amsart}
\usepackage{latexsym, amsmath, amssymb, amsthm, mathrsfs}
\usepackage[colorlinks=true,linkcolor=blue,urlcolor=blue, citecolor=blue]%
  {hyperref}
\usepackage[shortalphabetic]{amsrefs}
\usepackage[T1]{fontenc}
\usepackage{mathptmx}
\usepackage{microtype}

\usepackage[centering, includeheadfoot, hmargin=1in, vmargin=1in,
  headheight=30.4pt]{geometry}
\newtheorem{lemma}{Lemma}[section]
\newtheorem{theorem}[lemma]{Theorem}
\newtheorem{corollary}[lemma]{Corollary}
\newtheorem{proposition}[lemma]{Proposition}
\newtheorem{procedure}[lemma]{Procedure}
\theoremstyle{definition}
\newtheorem{remark}[lemma]{Remark}
\newtheorem{observation}[lemma]{Observation}
\newtheorem{example}[lemma]{Example}
\newcommand{\define}[1]{{\bfseries\itshape #1}}

\renewcommand{\theequation}%
{\arabic{section}.\arabic{lemma}.\arabic{equation}}

\renewcommand{\AA}{\ensuremath{\mathbb{A}}} 
\newcommand{\CC}{\ensuremath{\mathbb{C}}} 
\newcommand{\NN}{\ensuremath{\mathbb{N}}}
\newcommand{\PP}{\ensuremath{\mathbb{P}}} 
\newcommand{\RR}{\ensuremath{\mathbb{R}}} 
 
\newcommand{\ZZ}{\ensuremath{\mathbb{Z}}} 
\renewcommand{\geq}{\geqslant}
\renewcommand{\leq}{\leqslant}
\DeclareMathOperator{\codim}{codim}
\DeclareMathOperator{\cone}{cone}
\DeclareMathOperator{\conv}{conv}

\DeclareMathOperator{\Ker}{Ker}
\DeclareMathOperator{\Image}{Im}
\DeclareMathOperator{\New}{New}
\DeclareMathOperator{\Pos}{P}
\DeclareMathOperator{\Proj}{Proj}
\DeclareMathOperator{\sign}{sgn}
\DeclareMathOperator{\Sos}{\Sigma}
\DeclareMathOperator{\Spec}{Spec}
\DeclareMathOperator{\Sym}{Sym}
\DeclareMathOperator{\variety}{V}
\DeclareMathOperator{\hh}{h}

\begin{document}

\title[Sums of squares and minimal degree]{Sums of squares and varieties of
  minimal degree}

\author[G.~Blekherman]{Grigoriy Blekherman}
\address{Greg Blekherman\\ School of Mathematics \\ Georgia Tech, 686 Cherry
  Street\\ Atlanta\\ GA\\ 30332\\ USA}
\email{\href{mailto:greg@math.gatech.edu}{greg@math.gatech.edu}}

\author[G.G.~Smith]{Gregory G. Smith}
\address{Gregory G. Smith\\ Department of Mathematics and Statistics \\ Queen's
  University\\ Kingston \\ ON \\ K7L~3N6\\ Canada}
\email{\href{mailto:ggsmith@mast.queensu.ca}{ggsmith@mast.queensu.ca}}

\author[M.~Velasco]{Mauricio Velasco} \address{Mauricio Velasco\\ Departamento
  de Matem\'aticas\\ Universidad de los Andes\\ Carrera 1 No. 18a 10\\ Edificio
  H\\ Primer Piso\\ 111711 Bogot\'a\\ Colombia}
\email{\href{mailto:mvelasco@uniandes.edu.co}{mvelasco@uniandes.edu.co}}


\date{3 August 2013}

\begin{abstract}
  Let $X \subseteq \PP^n$ be a real nondegenerate subvariety such that the set
  $X(\RR)$ of real points is Zariski dense.  We prove that every real quadratic
  form that is nonnegative on $X(\RR)$ is a sum of squares of linear forms if
  and only if $X$ is a variety of minimal degree.  This substantially extends
  Hilbert's celebrated characterization of equality between nonnegative forms
  and sums of squares.  We obtain a complete list for the cases of equality and
  also a classification of the lattice polytopes $Q$ for which every nonnegative
  Laurent polynomial with support contained in $2Q$ is a sum of squares.
\end{abstract}

\maketitle

\section{Introduction} 
\label{sec:intro}

\noindent
The study of nonnegativity and its relation with sums of squares is a basic
challenge in real algebraic geometry.  The classification of varieties of
minimal degree is one of the milestones of classical complex algebraic geometry.
The goal of this paper is to establish the deep connection between these
apparently separate topics.

To achieve this, let $X \subseteq \PP^n$ be an embedded real projective variety
with homogeneous coordinate ring $R$.  The variety $X$ has minimal degree if it
is nondegenerate (i.e.\ not contained in a hyperplane) and $\deg(X) = 1 +
\codim(X)$.  An element $f \in R$ is nonnegative if its evaluation at each real
point of $X$ is at least zero.  Our main theorem is a broad generalization of
Hilbert's 1888 classification of nonnegative forms and provides a tight
connection between real and complex algebraic geometry.

\begin{theorem}
  \label{thm:main}
  Let $X \subseteq \PP^n$ be a real irreducible nondegenerate projective
  subvariety such that the set $X(\RR)$ of real points is Zariski dense.  Every
  nonnegative real quadratic form on $X$ is a sum of squares of linear forms if
  and only if $X$ is a variety of minimal degree.
\end{theorem}

\noindent
Using the Veronese embedding, this theorem extends to forms of any even degree
(see Remark~\ref{rem:higherdegree}).  

Together with the well-known catalogue for varieties of minimal degree (e.g.\
Theorem~1 in \cite{EisenbudHarris}), our main theorem produces a complete list
of varieties for which nonnegative quadratic forms are sums of squares.  There
are exactly three families:
\begin{itemize}
\item totally-real irreducible quadratic hypersurfaces (see
  Example~\ref{exa:hypersurface}), 
\item cones over the Veronese surface (see Example~\ref{exa:VeroneseSurface}),
  and 
\item rational normal scrolls (see Example~\ref{exa:scroll}).
\end{itemize}

By replacing elements of $R$ with global sections of a line bundle, we develop
an intrinsic version of the main theorem (see Theorem~\ref{thm:linearseries}).
Applying this to line bundles on projective space, we recover Hilbert's
classification of nonnegative forms in a standard graded polynomial ring---for
binary forms, quadratic forms, and ternary quartics, nonnegativity is equivalent
to being a sum of squares and, in all other situations, there exists nonnegative
forms that is not a sum of squares (see Example~\ref{exa:Hilbert}).  In
particular, the exceptional Veronese surface corresponds to the exceptional case
of ternary quartics.  We obtain the classification of multiforms appearing in
\cite{CLR} from line bundles on a product of projective spaces (see
Example~\ref{exa:CLR}).  More generally, by working with a projective toric
variety or a multigraded polynomial ring, we enumerate the cases in which every
nonnegative multihomogeneous polynomial may be expressed as a sum of squares.
Specifically, we discover that the ternary quartics belong to an infinite family
consisting of cones over the Veronese surface (see
Example~\ref{exa:VeroneseCone}) and all other cases come from rational normal
scrolls (see Example~\ref{exa:smoothScroll} and Remark~\ref{rem:nonsmooth}).

Enhancing the intrinsic approach for line bundles on a toric variety yields an
analogue of our main theorem for sparse Laurent polynomials.  To be more
precise, let $M$ be an affine lattice of rank $m$ and let $Q$ be an
$m$-dimensional lattice polytope in $M \otimes_{\ZZ} \RR$.  The
$\hh^*$-polynomial of $Q$ is defined by
\[
\hh^*_0(Q) + \hh^*_1(Q) \, t + \dotsb + \hh^*_m(Q) \, t^m = (1-t)^{m+1}
\sum\nolimits_{k \geq 0} |(kQ) \cap M| \, t^k \, .
\]
We establish that every nonnegative Laurent polynomial with Newton polytope in
$2Q$ is a sum of squares if and only if $\hh^*_2(Q) = 0$ and the image of the
real points under the associate morphism is dense in the strong topology (see
Theorem~\ref{thm:sparse}).  We also describe all of the lattice polytopes $Q$
for which $\hh^*_2(Q) = 0$ (see Proposition~\ref{pro:h^*_2=0}).  This
generalizes the main theorem in \cite{BN} classifying degree-one lattice
polytopes (see Remark~\ref{rem:BN}).

For the proof of Theorem~\ref{thm:main}, convexity provides the bridge between
real and complex algebraic geometry.  The collections of nonnegative elements
and sums of squares both form closed convex cones (see Lemma~\ref{lem:cones}).
More significantly, the dual of the sums-of-squares cone is a spectrahedron, so
its extremal rays have an algebraic characterization (see
Observation~\ref{ob:extreme}).  This characterization drives the transition
between real and complex algebraic geometry.

\subsection*{Contents of the Paper} 

Section~\ref{sec:convex} defines the fundamental cones: $\Pos_X$ consists of the
nonnegative elements and $\Sos_X$ consists of the sums of squares.  The
description in Lemma~\ref{lem:hsop} of the extremal rays of $\Sos^*_X$ is the
key.  In Section~\ref{sec:non-sos}, we introduce the quadratic deficiency
$\varepsilon(X)$ of the embedded variety $X \subseteq \PP^n$.  This numerical
invariant is an algebraic incarnation of $\hh^*_2(Q)$ and forms the pivotal link
between quadratic forms and varieties of minimal degree; see
Lemma~\ref{lem:deficiency}.  As Proposition~\ref{pro:onlyif} establishes, having
$\varepsilon(X) > 0$ is a sufficient condition for the existence of nonnegative
real quadratic forms on $X$ that cannot be expressed a sums of
squares. Procedure~\ref{proc:hilbert} constructs nonnegative quadratic forms
that are not sums of squares.  Proposition~\ref{pro:minimal} analyzes the
varieties with $\varepsilon(X) = 1$.  We prove the main theorem in
Section~\ref{sec:equality}.  Proposition~\ref{pro:if} shows that $\varepsilon(X)
= 0$ is sufficient.  Remark~\ref{rem:truncated} connects the main theorem to the
truncated moment problem in real analysis.  Section~\ref{sec:intrinsic}
translates the main theorem and principal examples into the intrinsic setting of
a variety with a basepoint-free linear series.  Lastly, Section~\ref{sec:sparse}
develops the polyhedral theory.

\subsection*{Acknowledgements}

We thank Bernd Sturmfels for stimulating our interest in convex algebraic
geometry.  We also thank Matthias Beck, Mircea Musta\c{t}\u{a}, and Mike Roth
for helpful conversations.  The first author was partially supported by a Sloan
Fellowship, NSF grant DMS-0757212, the Mittag-Leffler Institute, and IPAM; the
second author was partially supported by NSERC, the Mittag-Leffler Institute,
and MSRI; and the third author was partially supported by the FAPA grants from
Universidad de los Andes.

\section{Convexity and Spectrahedral Properties}
\label{sec:convex}

\noindent
In this section, we develop the necessary tools from convex algebraic geometry.
We carefully define the fundamental cones and highlight their properties.

Let $X \subseteq \PP^n$ be a nondegenerate $m$-dimensional totally-real
projective subvariety.  In particular, $X$ is a geometrically integral
projective scheme over $\Spec(\RR)$ such that $X$ is not contained in a
hyperplane and the set $X(\RR)$ of real points is Zariski dense.  Set $e := n-m
= \codim(X)$.  If $I$ is the unique saturated homogeneous ideal vanishing on
$X$, then the $\ZZ$-graded coordinate ring of $X$ is $R := \RR[x_0,\dotsc,
x_n]/I$.  For each $j \in \ZZ$, the graded component $R_j$ of degree $j$ is a
finite dimensional real vector space.  Since $X$ is nondegenerate, we have
$\RR[x_0,\dotsc, x_n]_1 = R_1$.  Given $f \in R_{2j}$ and $p \in X(\RR)$, the
\define{sign of $f$ at $p$} is $\sign_{p}(f) := \sign \bigl(
\tilde{f}(\tilde{p}) \bigr) \in \{-1,0,1\}$ where the polynomial $\tilde{f} \in
\RR[x_0, \dotsb, x_m]_{2j}$ maps to $f$ and the nonzero real point $\tilde{p}
\in \AA^{n+1}(\RR)$ maps to $p$ under the canonical quotient homomorphisms (cf.\
\S2.4 in \cite{Scheiderer}).  Since $p \in X(\RR)$, the real number
$\tilde{f}(\tilde{p})$ is independent of the choice $\tilde{f}$.  Similarly, the
choice of the affine representative $\tilde{p}$ is determined up to a nonzero
real number, so the value of $\tilde{f}(\tilde{p})$ is determined up to the
square of a nonzero real number because the degree of $f$ is even.  We simply
write $f(p) \geq 0$ for $\sign_{p}(f) \geq 0$.

The central objects of study are the following subsets in $R_2$:
\begin{alignat*}{2}
  \Pos_{X} &:= \{ f \in R_2 : \text{$f(p) \geq 0$ for all $p \in X(\RR)$} \}
  \, , &\text{and}& \\
  \Sos_{X} &:= \{ f \in R_2 : \text{there exists $g_1, g_2, \dotsc, g_k \in R_1$
    such that $f = g_1^2 + g_2^2 + \dotsb + g_k^2$} \} \, .
\end{alignat*}
We clearly have $\Sos_X \subseteq \Pos_X$.  To describe the properties of these
subsets, consider the $\RR$-linear map $\sigma \colon \Sym^2(R_1) \to R_2$
induced by multiplication in $R$ and let $\sigma^* \colon R_2^* \to
\Sym^2(R_1^*) = \bigl( \Sym^2(R_1) \bigr)^*$ be the dual.  More explicitly, for
a linear functional $\ell \in R_2^*$, $\sigma^*(\ell)$ is the symmetric bilinear
map $R_1 \otimes_\RR R_1 \to \RR$ defined by $g_1 \otimes g_2 \mapsto
\ell(g_1g_2)$.  For $p \in X(\RR)$, evaluation at any affine representative
$\tilde{p} \in \AA^{n+1}(\RR)$ determines $\tilde{p}^* \in R_1^*$.  Because $p
\in X(\RR)$, the map $\Sym^2(R_1) \to \RR$ induced by $\tilde{p}^* \in R_1^*$
annihilates $I_2$ and defines the element $(\tilde{p}^*)^2 \in R_2^*$.  Since
evaluations at distinct representatives differ by the square of a nonzero
constant, the ray $\cone\bigl( (\tilde{p}^*)^2 \bigr) := \{ \lambda \cdot
(\tilde{p}^*)^2 : \lambda \geq 0 \} \subseteq R_2^*$ is independent of the
choice of the affine representative.

The following fundamental lemma is a minor variant of well-known results (cf.\
Theorem~3.35 in \cite{Lau} or Exercise~4.2 in \cite{BPT}).

\begin{lemma}
  \label{lem:cones}
  Both $\Pos_X$ and $\Sos_X$ are pointed full-dimensional closed convex cones in
  the real vector space $R_2$.  We also have 
  \begin{align*}
    \Pos_X^* &= \cone\bigl( (\tilde{p}^*)^2 : p \in X(\RR) \bigr) = \{
    \lambda_1^{} (\tilde{p}_1^*)^2 + \lambda_2^{} (\tilde{p}_2^*)^2 + \dotsb +
    \lambda_k^{} (\tilde{p}_k^*)^2 : \text{$\tilde{p}_i \in X(\RR)$ and
      $\lambda_i^{} \geq 0$}
    \} \, ,  \\
    \Sos_X^* &= \{ \ell \in R_2^* : \text{$\sigma^*(\ell)$ is
      positive-semidefinite} \} \, .
  \end{align*}
\end{lemma}

\begin{proof}
  We first consider the nonnegative elements.  Set $P := \cone\bigl(
  (\tilde{p}^*)^2 : p \in X(\RR) \bigr)$.  By definition, an element $f \in R_2$
  belongs to $\Pos_X$ if and only if $f(p) \geq 0$, so $P^* = \Pos_X$.  It
  follows that $\Pos_X$ is a closed convex cone and $(P^*)^* = \Pos_X^*$.  To
  show that $P$ is closed, fix an inner product on $R_2^*$ and let $\ell \mapsto
  \| \ell \|$ denote the associated norm.  For each $p \in X(\RR)$, the linear
  functional $\frac{ (\tilde{p}^*)^2 }{\| (\tilde{p}^*)^2 \|} \in R_2^*$ is
  independent of the choice of the affine representative.  Since $X(\RR)
  \subseteq \PP^n(\RR)$ is compact in the induced metric topology, the spherical
  section $K := \left\{ \frac{ (\tilde{p}^*)^2 }{\| (\tilde{p}^*)^2 \|} : p \in
    X(\RR) \right\}$ of $P$ is compact.  Because $X$ is totally-real, the convex
  hull of $K$ does not contain $0$.  Since $P$ is the conical hull of $K$, the
  cone $P$ is closed and $P = \Pos_X^*$.  By hypothesis, the set $X(\RR)$ of
  real points is Zariski dense, so $\Pos_X$ cannot contain a nonzero linear
  subspace.

  We next examine the sums of squares.  For $\ell \in \Sos_X^*$, we have
  $\ell(f^2) \geq 0$ for all $f \in R_1$, so the bilinear symmetric form
  $\sigma^*(\ell)$ is positive semidefinite.  Conversely, if $\sigma^*(\ell)$ is
  positive semidefinite, then $\ell(g^2) \geq 0$ for all $g \in R_1$.  Hence, we
  have $\ell(g_1^2 + g_2^2 + \dotsb +g_k^2) = \ell(g_1^2) + \ell(g_2^2) + \dotsb
  + \ell(g_k^2) \geq 0$ for $g_1, g_2, \dotsc, g_k \in R_1$, and $\ell \in
  \Sos_X^*$.  Thus, $\ell \in \Sos_X^*$ if and only $\sigma^*(\ell)$ is a
  positive-semidefinite symmetric bilinear form.  By duality, the cone $\Sos_X$
  is a linear projection of the convex cone $\mathbb{S}_+$ of
  positive-semidefinite symmetric bilinear forms.  Since $\mathbb{S}_+$ is
  full-dimensional and $\sigma \colon \Sym^2(R_1) \to R_2$ is surjective, it
  follows that $\Sos_X$ is also full-dimensional.  To complete the proof, fix an
  inner product on $R_1$ and let $g \mapsto \|g \|$ denote the associated norm.
  The spherical section $K' := \{ g^2 \in R_2 : \text{$g \in R_1$ satisfies $\|
    g\| =1$} \}$ is compact, because it is the continuous image of a compact
  set.  As above, its convex hull does not contain the origin.  Therefore, the
  cone $\Sos_X$ is closed.
\end{proof}

The subsequent observation is the key insight from convex geometry needed to
prove our main result.  Lemma~\ref{lem:hsop} is the simple, but crucial,
algebraic consequence of this observation.

\begin{observation}
  \label{ob:extreme}
  Lemma~\ref{lem:cones} shows that $\Sos_X^*$ is a spectrahedron, that is a
  section of the convex cone $\mathbb{S}_{+}$ of positive-semidefinite symmetric
  bilinear forms.  Hence, Theorem~1 in \cite{RamanaGoldman} implies that every
  face of $\Sos_X^*$ is exposed.  The unique face containing $\ell \in \Sos_X^*$
  in its relative interior is given by $H_\ell \cap \Sos_X^*$ where $H_\ell :=
  \{ \ell' \in R_2^* : \Ker \bigl( \sigma^*(\ell) \bigr) \subseteq \Ker \bigl(
  \sigma^*(\ell') \bigr) \}$.  Moreover, Corollary~3 in \cite{RamanaGoldman}
  characterizes the extremal rays as follows: a point in a spectrahedron is
  extremal if and only if the kernel of its associated positive semidefinite
  form is maximal with respect to the inclusion.  Hence, if $\ell \in \Sos_X^*$
  is an extremal point and $A \in \Image(\sigma^*)$ such that $\Ker\bigl(
  \sigma^*(\ell) \bigr) \subseteq \Ker(A)$, then we have $\sigma^*(\ell) =
  \lambda A$ for some $\lambda \in \RR$.
\end{observation}

\begin{lemma}
  \label{lem:hsop}
  If $\ell \in R_2^*$ generates an extremal ray of $\Sos_X^*$, then either
  $\ell$ is given by evaluation at some $p \in X(\RR)$ or the subspace $\Ker
  \bigl( \sigma^*(\ell) \bigr) \subseteq R_1$ contains a homogeneous system of
  parameters on $R$.
\end{lemma}

\begin{proof}
  First, suppose that the linear forms in $\Ker \bigl( \sigma^*(\ell) \bigr)$
  have a common real zero $p \in X(\RR)$.  Choose an affine representative
  $\tilde{p} \in \AA^{n+1}(\RR)$.  If $\sigma^* \bigl( (\tilde{p}^*)^2 \bigr)
  \in \Sym^2(R_1^*)$ is the associated symmetric form, then we have $\Ker \bigl(
  \sigma^*(\ell) \bigr) \subseteq \Ker \bigl( \sigma^* \bigl( (\tilde{p}^*)^2
  \bigr) \bigr)$.  Since $\ell \in \Sos_X^*$ generates an extremal ray,
  Observation~\ref{ob:extreme} implies that $\sigma^*(\ell) = \lambda
  (\tilde{p}^*)^2$ for some $\lambda \in \RR$.  As both $\sigma^*(\ell)$ and
  $(\tilde{p}^*)^2$ are positive semidefinite, it follows that $\lambda > 0$.
  Hence, by changing the affine representative for $p \in X(\RR)$ to
  $\sqrt{\lambda} \tilde{p} \in \AA^{n+1}(\RR)$, we obtain $\ell =
  (\tilde{p}^*)^2$.

  Now, assume that the only common zeroes for the linear forms in $\Ker \bigl(
  \sigma^*(\ell) \bigr)$ have a nonzero complex part.  Choose an affine
  representative $\tilde{\zeta} \in \AA^{n+1}(\CC)$ for one of these complex
  zeroes.  Define $\ell' \in R_2^*$ by $\ell'(f) := \operatorname{Re}\bigl(
  f(\tilde{\zeta}) \bigr)$ to be the real part of the evaluation of $f$ at
  $\tilde{\zeta}$; this is well-defined because $\zeta \in X$.  By construction,
  we have $\Ker \bigl( \sigma^*(\ell) \bigr) \subseteq \Ker \bigl(
  \sigma^*(\ell') \bigr)$.  Since $\ell \in \Sos_X^*$ generates an extremal ray,
  Observation~\ref{ob:extreme} implies that $\sigma^*(\ell) = \lambda
  \sigma^*(\ell')$ for some $\lambda \in \RR$.  However, there exist $g_1, g_2
  \in R_1$ such that $g_1(\tilde{\zeta}) = 1$ and $g_2(\tilde{\zeta}) =
  \sqrt{-1}$, so $\ell'(g_1^2) = 1$ and $\ell'(g_2^2) = -1$.  Hence,
  $\sigma^*(\ell)$ is not positive semidefinite, which by Lemma~\ref{lem:cones}
  contradicts the hypothesis that $\ell \in \Sos_X^*$.  In other words, our
  assumption guarantees that the linear forms in $\Ker \bigl( \sigma^*(\ell)
  \bigr)$ have no common zeroes in $X$.  Therefore, we conclude that $\Ker \bigl(
  \sigma^*(\ell) \bigr) \subseteq R_1$ contains a homogeneous system of
  parameters via the Nullstellensatz.
\end{proof}

\section{Separating the Fundamental Cones}
\label{sec:non-sos}

\noindent
This section investigates differences between the sums-of-squares cone $\Sos_X$
and the nonnegative cone $\Pos_X$.  It relates the positivity of an algebraic
invariant associated to an embedded variety $X \subseteq \PP^n$ with the proper
inclusion of $\Sos_X$ in $\Pos_X$.  We construct witnesses that separate
$\Sos_X$ and $\Pos_X$.  Moreover, we give a general procedure for constructing
nonnegative real quadratic forms on $X$ that are not sums of squares.

Emulating \S5 in \cite{Zak}, we define the \define{quadratic deficiency} of the
subvariety $X \subseteq \PP^n$ to be $\varepsilon(X) := \binom{e+1}{2} - \dim
(I_2)$ where $e := \codim(X)$ and $I$ is the unique saturated homogeneous ideal
vanishing on $X$.  The first lemma provides a couple elementary
reinterpretations for this numerical invariant and recounts the important
connection between $\varepsilon(X)$ and varieties of minimal degree.

\begin{lemma}
  \label{lem:deficiency}
  The quadratic deficiency $\varepsilon(X)$ equals the coefficient of the
  quadratic term in the numerator of the Hilbert series for $X$ and
  $\varepsilon(X) = \dim (R_2) - (m+1)(n+1) + \binom{m+1}{2}$.  Moreover,
  $\varepsilon(X)$ is nonnegative and we have $\varepsilon(X) = 0$ if and only
  $\deg(X) = 1 + \codim(X)$.
\end{lemma}

\begin{proof}
  Since $X$ is nondegenerate, we have $\dim(R_0) = 1$ and $\dim(R_1) = n+1$.
  Hence, there exists a polynomial $1 + e \, t + \hh^*_2(X) \, t^2 + \dotsb +
  \hh^*_n(X) \, t^{n} \in \ZZ[t]$ such that
  \begin{align*}
    \sum_{j \geq 0} \dim (R_j) \, t^j &= \frac{1 + e \, t + \hh^*_2(X) \, t^2 +
      \dotsb + \hh^*_{n}(X) \, t^{n}}{(1-t)^{m+1}} \, .
  \end{align*}
  Using the binomial theorem to compare the coefficients of the degree-two
  terms, we obtain
  \begin{align*}
    \dim(R_2) &= \tbinom{m+2}{2} + e \tbinom{m+1}{1} + \hh^*_2(X)
    \tbinom{m+0}{0} = \tbinom{m+1}{2} + \tbinom{m+1}{1} + (n-m)(m+1) +
    \hh^*_2(X) \\
    &= \tbinom{m+1}{2} -m(m+1) + (n+1)(m+1) + h^*_2 = - \tbinom{m+1}{2} +
    (n+1)(m+1) + \hh^*_2(X) \, .
  \end{align*}
  Rearranging this equation and using the presentation for $R$ yields
  \begin{alignat*}{2}
    \hh^*_2(X) &= \dim (R_2) - (m+1)(n+1) + \tbinom{m+1}{2} = \tbinom{n+2}{2} -
    \dim (I_2) -(m+1)(n+1) + \tfrac{(m+1)m}{2} \\
    &= \tfrac{(n-m+1)(n-m)}{2} - \dim(I_2) = \tbinom{e+1}{2} - \dim(I_2) =
    \varepsilon(X) \, ,
  \end{alignat*}
  which establishes the results in the first sentence of the lemma.  Both parts
  of the second sentence are well-known.  As Theorem~1.2 in \cite{Lvovsky}
  indicates, they can be deduced from Castelnuovo's Lemma, which states that if
  $n(n-1)/2$ linearly independent quadrics pass through at least $2n+3$ points
  in linearly general position in $\PP^n$, then these points lie on a rational
  normal curve.  Corollary~5.4 and Corollary~5.8 in \cite{Zak} give alternative
  proofs using properties of secant varieties.
\end{proof}

The subsequent proposition extends both Theorem~1.1 and Theorem~1.2 in
\cite{Blekherman} and provides one of the implications needed for the proof of
Theorem~\ref{thm:main}.

\begin{proposition}
  \label{pro:onlyif}
  If $\varepsilon(X) > 0$, then $\Sos_X$ is a proper subset of $\Pos_X$.
\end{proposition}

\begin{proof}
  Since $\varepsilon(X) > 0$, it follows from Lemma~\ref{lem:deficiency} that
  $\deg(X) > 1 + \codim(X)$.  We begin by showing that there exists $h_1, h_2,
  \dotsc, h_m \in R_1$ such that $Z := X \cap \variety(h_1, h_2, \dotsc, h_m)$
  is a reduced set of points in linearly general position containing at least $e
  + 1$ distinct real points.  To achieve this, observe that B\'ezout's Theorem
  implies that the intersection of a positive-dimensional irreducible
  nondegenerate variety with a general hyperplane is nondegenerate; see
  Proposition~18.10 in \cite{Harris}.  Next, Bertini's Theorem (e.g.\
  Th\'eor\`eme~6.3 in \cite{J}) establishes that a general hyperplane section of
  a geometrically integral variety of dimension at least 2 is geometrically
  integral and that a general hyperplane section of a geometrically reduced
  variety is geometrically reduced.  Thirdly, we see that a geometrically
  integral real variety is totally real if and only if it contains a nonsingular
  real point; see \S1 in \cite{Becker}.  Finally, we note that the locus of
  hyperplanes that intersect the nonsingular locus of $X$ transversely contains
  a nonempty Zariski open set.  By combining these four observations, we deduce
  that the intersection of $X$ with $m-1$ general hyperplanes yields a
  nondegenerate geometrically integral totally-real curve $C$ in $\variety(h_1)
  \cap \variety(h_2) \cap \dotsb \cap \variety(h_{m-1}) \cong \PP^{e+1}$.  The
  degree of $C$, which equals $\deg(X)$, is at least $e + 1$; see
  Corollary~18.12 in \cite{Harris}.  Any set of $e+1$ distinct real points on
  $C$ lie in a real hyperplane.  Since $C$ is nondegenerate and totally-real,
  the locus of hyperplanes intersecting $C$ in at least $e + 1$ distinct real
  points has dimension at least $e + 1$.  Hence, there exists a hyperplane
  $\variety(h_m)$ such that intersection with $C$ is a set of points in linearly
  general position containing at least $e + 1$ distinct real points.

  To complete the proof, we use points in $Z$ to exhibit a linear functional in
  $\Sos_X^* \setminus \Pos_X^*$.  We divide the analysis into two cases.  In the
  first case, we assume that the intersection $Z$ contains at least $e+2$
  distinct real points.  Choose an affine representative $\tilde{p}_j$ where $1
  \leq j \leq e+2$ for each of these points.  The points lie in $\variety(h_1)
  \cap \variety(h_2) \cap \dotsb \cap \variety(h_m) \cong \PP^{e}$, so the
  evaluations $\tilde{p}_j^*$ satisfy a linear equation in $R_1^*$.  The
  coefficients in this linear equation are nonzero and determine a unique point
  in $\PP^{e+1}$ because $p_1, \dotsc, p_{e+2}$ are in linearly general
  position.  Specifically, there are unique nonzero $\lambda_1, \lambda_2,
  \dotsc, \lambda_{e+1} \in \RR$ such that
  \begin{align}
    \label{eq:firstCB}
    0 &= \lambda_1^{\,} \tilde{p}_1^* + \lambda_2^{\,} \tilde{p}_2^* + \dotsb +
    \lambda_{e+1} \tilde{p}_{e+1}^* + \tilde{p}_{e+2}^* \, .
  \end{align}
  Fix $\kappa_j > 0$ for $1 \leq j \leq e+1$, set $\kappa_{e+2} := \bigl(
  \frac{\lambda_1^2}{\kappa_1} + \frac{\lambda_2^2}{\kappa_2} + \dotsb +
  \frac{\lambda_{e+1}^2}{\kappa_{e+1}} \bigr)^{-1}$, and consider
  \[
  \ell := \kappa_1^{} (\tilde{p}_1^*)^2 + \kappa_2^{} (\tilde{p}_2^*)^2 + \dotsb
  + \kappa_{e+1}^{} (\tilde{p}_{e+1}^*)^2 - \kappa_{e+2}^{}
  (\tilde{p}_{e+2}^*)^2 \in R_2^* \, .
  \]  
  Since $\kappa_j > 0$ for all $1 \leq j \leq e+1$, equation \eqref{eq:firstCB}
  yields
  \begin{align*}
    \ell &= \sum_{j=1}^{e+1} \bigl( \sqrt{\kappa_j^{}} \tilde{p}_j^* \bigr)^2 -
    (\tilde{p}_{e+2}^*)^2 \Bigl( \sum_{j=1}^{e+1} \bigl(
    \tfrac{\lambda_j}{\sqrt{\kappa_j}} \bigr)^{\! 2} \Bigr)^{\! -1} \\ &= \Bigl(
    \sum_{j=1}^{e+1} \bigl( \tfrac{\lambda_j}{\sqrt{\kappa_j}} \bigr)^{\! 2}
    \Bigr)^{\! -1} \left[ \Bigl( \sum_{j=1}^{e+1} \bigl(
      \tfrac{\lambda_j}{\sqrt{\kappa_j}} \bigr)^{\! 2} \Bigr) \Bigl(
      \sum_{j=1}^{e+1} \bigl( \sqrt{\kappa_j^{}} \tilde{p}_j^* \bigr)^{\! 2}
      \Bigr) - \Bigl( \sum_{j=1}^{e+1} \lambda_j^{} \tilde{p}_j^* \Bigr)^{\!\!
        2} \right] \, .
  \end{align*}
  Hence, the Cauchy-Schwartz inequality shows that $\ell$ is nonnegative on
  squares, whence $\ell \in \Sos_X^*$ by Lemma~\ref{lem:cones} (cf.\ Theorem~6.1
  in \cite{Blekherman}).  Nevertheless, there exists $g \in R_1$ such that
  $\tilde{p}_j^*(g) = g(\tilde{p}_j) = \lambda_j^{\,} \kappa_j^{-1}$ for all $1
  \leq j \leq e+1$, which implies that $\ell(g^2) = 0$.  In addition, choose the
  $\kappa_j$ for $1 \leq j \leq e+1$ so that $g$ does not vanish at any point in
  $Z$.  Since $g^2 +h_1^2 + h_2^2 + \dotsb + h_m^2$ is strictly positive on $X$
  and $\ell(g^2 + h_1^2 + \dotsb + h_m^2) = 0$, the linear functional $\ell$
  cannot be a nonnegative combination of points evaluations at $X(\RR)$.
  Therefore, we have $\ell \in \Sos_X^* \setminus \Pos_X^*$.

  In the second case, we assume that $Z$ has at most $e+1$ distinct real points.
  Since $\deg(X) \geq e+1$, the reduced set $Z$ contains at least one pair of
  complex conjugate points.  Let $\tilde{a} \pm \tilde{b} \sqrt{-1} \in
  \AA^{n+1}(\CC)$, where $\tilde{a}, \tilde{b} \in \AA^{n+1}(\RR)$, be affine
  representatives for such a pair and choose an affine representative
  $\tilde{p}_j$ for $1 \leq j \leq e$ for some real points in $Z$.  As in the
  other case, the chosen $e+2$ points lie in $\variety(h_1) \cap \variety(h_2)
  \cap \dotsb \cap \variety(h_m) \cong \PP^{e}$, so the evaluations satisfy a
  linear equation in $R_1^*$.  Again, the coefficients are nonzero and determine
  a unique point $\PP^{e+1}$ because the points in $Z$ are in linearly general
  position.  Since the unique linear equation is invariant under conjugation,
  the coefficients are real and the coefficients of $(\tilde{a} + \tilde{b}
  \sqrt{-1})^*$ and $(\tilde{a} - \tilde{b} \sqrt{-1})^*$ are equal.
  Specifically, there are unique nonzero $\lambda_1, \lambda_2, \dotsc,
  \lambda_{e} \in \RR$ such that
  \begin{align} 
    \label{eq:secondCB}
    \begin{split}
      0 &= \lambda_1^{} \tilde{p}_1^* + \lambda_2^{} \tilde{p}_2^* + \dotsb +
      \lambda_{e}^{} \tilde{p}_{e}^* + \tfrac{1}{2}(\tilde{a} + \tilde{b}
      \sqrt{-1})^* + \tfrac{1}{2} (\tilde{a} - \tilde{b} \sqrt{-1})^* \\
      &= \lambda_1^{} \tilde{p}_1^* + \lambda_2^{} \tilde{p}_2^* + \dotsb +
      \lambda_{e}^{} \tilde{p}_{e}^* + \tilde{a}^*
    \end{split}
  \end{align}
  Taking the real and imaginary parts of $\bigl( (\tilde{a} \pm \tilde{b}
  \sqrt{-1})^* \bigr)^2 \in R_2^*$ yields the linear independent real
  functionals $(\tilde{a}^*)^2 - (\tilde{b}^*)^2 \in R_2^*$ and $2\tilde{a}^*
  \tilde{b}^* \in R_2^*$.  Fix $\kappa_j > 0$ for $1 \leq j \leq e$, choose
  $\kappa_{e+1}$ and $\kappa_{e+2}$ satisfying $(\kappa_{e+1}^2+
  \kappa_{e+2}^2)\kappa_{e+1}^{-1} := \bigl( \frac{\lambda_1^2}{\kappa_1} +
  \frac{\lambda_2^2}{\kappa_2} + \dotsb + \frac{\lambda_{e}^2}{\kappa_{e}}
  \bigr)^{-1}$, and consider
  \[
  \ell := \kappa_1^{} (\tilde{p}_1^*)^2 + \kappa_2^{} (\tilde{p}_2^*)^2 + \dotsb
  + \kappa_{e} (\tilde{p}_{e}^*)^2 - \kappa_{e+1}^{} \bigl( (\tilde{a}^*)^2 -
  (\tilde{b}^*)^2 \bigr) + \kappa_{e+2}^{} (2 \tilde{a}^* \tilde{b}^*) \in R_2^*
  \, .
  \]  
  Completing the square and using equation \eqref{eq:secondCB} yields
  \begin{align*}
    \ell &= \sum_{j=1}^{e} \bigl( \sqrt{\kappa_j^{}} \tilde{p}_j^* \bigr)^2 -
    \tfrac{\kappa_{e+1}^2+\kappa_{e+2}^2}{\kappa_{e+1}} (\tilde{a}^*)^2 +
    \kappa_{e+1}^{} \Bigl( \tilde{b}^* + \tfrac{\kappa_{e+2}}{\kappa_{e+1}}
    \tilde{a}^* \Bigr)^{\!\! 2} \\
    &= \Bigl( \sum_{j=1}^{e} \bigl( \tfrac{\lambda_j}{\sqrt{\kappa_j}}
    \bigr)^{\! 2} \Bigr)^{\! -1} \left[ \Bigl( \sum_{j=1}^{e} \bigl(
      \tfrac{\lambda_j}{\sqrt{\kappa_j}} \bigr)^{\! 2} \Bigr) \Bigl(
      \sum_{j=1}^{e} \bigl( \sqrt{\kappa_j^{}} \tilde{p}_j^* \bigr)^{\! 2} \Bigr) -
      \Bigl( \sum_{j=1}^{e} \lambda_j \tilde{p}_j^* \Bigr)^{\!\! 2} \right] +
    \kappa_{e+1}^{} \Bigl( \tilde{b}^* + \tfrac{\kappa_{e+2}}{\kappa_{e+1}}
    \tilde{a}^* \Bigr)^{\!\! 2} \, .
  \end{align*}
  Since we have $\kappa_{e+1} > 0$, the Cauchy-Schwartz inequality once more
  shows that $\ell$ is nonnegative on squares (cf.\ Theorem~7.1 in
  \cite{Blekherman}).  By repeating the argument above, we conclude that $\ell
  \in \Sos_X^* \setminus \Pos_X^*$.
\end{proof}

By enhancing the techniques used in the proof of Proposition~\ref{pro:onlyif},
we obtain a way to construct nonnegative polynomials that are not sums of
squares.  We describe this process below.  To make it computationally effective,
one needs an explicit bound for the coefficient $\delta$.

\begin{procedure}[Nonnegative polynomials that are not sums of squares]
  \label{proc:hilbert}
  Given an $m$-dimensional nondegenerate totally-real subvariety $X \subseteq
  \PP^n$ such that $\varepsilon(X) > 0$, the following steps yield a polynomial
  lying in $\Pos_X \setminus \Sos_X$.
  \begin{itemize}
  \item[Step~1:] Choose general linear forms $h_1, h_2, \dotsc, h_m \in R_1$
    which intersect in $\deg(X)$ distinct points in linearly general position
    where at least $e+1$ are real and smooth.  Fix $e$ smooth real points in the
    intersection and choose an additional linear form $h_0 \in R_1$ that
    vanishes only at the selected intersection points.  Let $L$ be the ideal in
    $R$ generated by $h_0, h_1, \dotsc, h_m$.
  \item[Step~2:] Choose a quadratic form $f \in R \setminus L^2$ that vanishes
    to order at least two at each of the selected intersection points.
  \item[Step~3:] For every sufficiently small $\delta > 0$, the polynomial
    $\delta f + h_0^2 + h_1^2 + \dotsb + h_m^2$ is nonnegative on $X$ but not a
    sum of squares.
  \end{itemize}
\end{procedure}

\begin{proof}[Correctness]  
  The existence of the $h_0, h_1, \dotsc, h_m$ in Step~1 follows from the first
  paragraph in the proof of Proposition~\ref{pro:onlyif}.  The quadratic forms
  in $L^2$ have dimension at most $\binom{m+2}{2}$.  Since second-order
  vanishing at $e$ distinct points imposes at most $(m+1)e$ linear conditions,
  Lemma~\ref{lem:deficiency} implies that the vector space of suitable $f$ has
  dimension at least
  \begin{align*}
    \dim(R_2) - (m+1)e - \tbinom{m+2}{2} &= \dim(R_2) - (m+1)\bigl( (n+1) -
    (m+1) \bigr) - \tbinom{m+2}{2} \\
    &= \dim(R_2) - (m+1)(n+1) + \tbinom{m+1}{2} = \varepsilon(X) \, ,
  \end{align*}
  which justifies Step~2.  For Step~3, suppose that $\delta f + h_0^2 + h_1^2 +
  \dotsb + h_m^2 = g_1^2 + g_2^2 + \dotsb + g_k^2$ for some $g_j \in R_1$.  It
  follows that each $g_j$ vanishes at the selected intersection points.  The
  ideal $L$ contains all linear forms which vanish at the selected intersection
  points, so $(g_j)^2 \in L^2$.  However, this gives a contradiction because $f
  \not\in L^2$.

  Hence, it remains to show that for a sufficiently small $\delta$, the
  polynomial $\delta f + h_0^2 + h_1^2 + \dotsb + h_m^2$ is nonnegative on $X$.
  Let $\tilde{X} \subseteq \AA^{n+1}(\RR)$ denote the affine cone of $X$ and let
  $\tilde{p}_1, \tilde{p}_2, \dotsc, \tilde{p}_e \in S^{n} \cap \tilde{X}$ be
  the affine representatives with unit length for the selected intersection
  points.  Since the selected points are nonsingular on $X$, the compact set
  $S^{n} \cap \tilde{X}$ is a real $m$-dimensional smooth manifold near each
  $\tilde{p}_j$ and the differentiable function $h_0^2 + h_1^2 + \dotsb + h_m^2$
  has a positive definite Hessian at the points $\tilde{p}_j$.  Since the
  $\tilde{p}_j$ are zeroes and critical points for the quadratic form $f$, it
  follows that there exists a $\delta_0 > 0$ and an neighbourhood $U_j$ of
  $\tilde{p_j}$ in $S^n \cap \tilde{X}$ for $1 \leq j \leq e$ such that $f$ is
  nonnegative on $U_j$.  On the compact set $K'' := ( S^n \cap \tilde{X})
  \setminus \bigcup_j U_j$, the function $h_0^2 + h_1^2 + \dotsb + h_m^2$ is
  strictly positive, so $\delta_1 := (\inf_{K''} h_0^2 + h_1^2 + \dotsb +
  h_m^2)/(\sup_{K''} |f|)$ is a strictly positive real number.  Hence, if $0 <
  \delta < \min(\delta_0, \delta_1)$, then $\delta f + h_0^2 + h_1^2 + \dotsb +
  h_m^2$ is nonnegative on $S^n \cap \tilde{X}$ and $X$.
\end{proof}

\begin{remark}
  In our context, Procedure~\ref{proc:hilbert} is a generalization of an idea
  going back to Hilbert.  To be more precise, let $\nu_d \colon \PP^{n}
  \rightarrow \PP^r$ with $r = \binom{n+d}{n}-1$ denote the $d$-th Veronese
  embedding of $\PP^n$.  For the subvarieties $\nu_3(\PP^2) \subset \PP^{10}$
  and $\nu_2(\PP^3) \subset \PP^{10}$, Hilbert~\cite{Hilbert} uses a similar
  procedure to prove the existence of nonnegative polynomials that are not sums
  of squares.  By working with concrete forms, Robinson uses this procedure to
  construct his celebrated form, see \S4b in \cite{Reznick}.  Again for
  $\nu_3(\PP^2) \subset \PP^{10}$ and $\nu_2(\PP^3) \subset \PP^{10}$,
  \cite{BIK} shows that the form $f$ in Procedure~\ref{proc:hilbert} is unique
  up to a constant multiple (i.e.\ the dimension estimates are sharp), and
  expresses it in terms of the intersection points of the $h_j$.
\end{remark}

In the simplest where $\Sos_X \neq \Pos_X$, namely $\varepsilon(X) = 1$, we can
clarify the difference between $\Sos_X$ and $\Pos_X$.  Proposition~5.10 in
\cite{Zak} shows that $\varepsilon(X) = 1$ if and only if $X$ is either a
hypersurface of degree $d \geq 3$, or a linearly normal variety such that
$\deg(X) = 2 + \codim(X)$ (a.k.a.\ a variety of almost minimal degree).  Given
$\ell \in R^*_2$, recall from Section~2 that $\sigma^*(\ell)$ is the
corresponding symmetric bilinear map.  Let $I(\ell)$ be the Gorenstein ideal in
$R$ generated by all homogeneous $g \in R$ such that either $\ell(fg) = 0$ for
all $f \in R_{2 - \deg(g)}$ or $\deg(g) > 2$.

\begin{proposition}
  \label{pro:minimal}
  Assume that $X$ is arithmetically Cohen-Macaulay and $\varepsilon(X) = 1$.  If
  $\ell \in \Sos^*_X$ is an extremal ray not contained in $\Pos_X^*$, then the
  quadratic form $\sigma^*(\ell)$ is positive semidefinite with $\dim \Ker\bigl(
  \sigma^*(\ell) \bigr) = m+1$.  Dually, if $f$ lies in the boundary of $\Sos_X$
  and not in the boundary of $\Pos_X$, then the element $f$ can be expressed as
  a sum of $m+1$ squares, but not as a sum of fewer squares.
\end{proposition}

\begin{proof}
  Lemma~\ref{lem:hsop} asserts that the subspace $\Ker \bigl( \sigma^*(\ell)
  \bigr) \subseteq R_1$ contains a homogeneous system of parameters $h_0, h_1,
  \dotsc, h_m$ on $R$.  Since $R$ is Cohen-Macaulay, this system of parameters
  is a regular sequence.  On the other hand, Remark~4.5 in \cite{BS} establishes
  that a projective variety of almost minimal degree is arithmetically
  Cohen-Macaulay if and only if it is arithmetically Gorenstein.  Hence, the
  quotient ring $R' := R/(h_0, h_2, \dotsc, h_m)$ is Gorenstein.
  Lemma~\ref{lem:deficiency} implies that the Hilbert function of $R'$ is
  $(1,e,1)$.  The ideal generated by the image of $I(\ell)$ in $R'$ under the
  canonical map is either trivial or contains the socle.  By definition, the
  elements in $I(\ell)_2$ are annihilated by $\ell$, so the second possibility
  cannot occur.  Hence, we have $I(\ell) = (h_0, h_2, \dotsc, h_m)$ and $\dim
  \Ker\bigl( \sigma^*(\ell) \bigr) = m+1$.

  If $f = g_1^2 + g_2^2 + \dotsb + g^2_k$ lies in the boundary of $\Sos_X$, then
  there exists an extremal ray $\ell \in \Sos^*_X$ such that $\ell(f) = 0$, so
  $(g_1, g_2, \dotsc, g_k) \subseteq \Ker \bigl( \sigma^*(\ell) \bigr)$.  Since
  $f$ is not in the boundary of $\Pos_X$, the element $f$ is strictly positive
  on $X(\RR)$ and $\ell$ is not defined by evaluation at a point.  The previous
  paragraph proves that $\dim \Ker\bigl( \sigma^*(\ell) \bigr) = m+1$ and this
  ensures that $f$ is a sum of at most $m+1$ squares.  To finish the proof,
  suppose that $f = g_1^2 + g_2^2 + \dotsb + g_k^2$ where $k \leq m$ and $g_1,
  g_2, \dotsc, g_k$ are linearly independent.  If $k < m$, then choose general
  linear forms $g_{k+1}, g_{k+2}, \dotsc, g_m$ in $\Ker \bigl( \sigma^*(\ell)
  \bigr)$.  Since $f$ strictly positive on $X(\RR)$, the ideal $J$ generated by
  $g_1, g_2, \dotsc, g_m$ defines a subscheme of $X$ that has no real zeroes.
  By perturbing $J$ if necessary, we obtain a subvariety $Z$ of $X$ that
  consists of $\deg(X)$ reduced points none of which are real.  Every element of
  $R_2$ vanishing at all the points in $Z$ lies in $\Ker \bigl( \sigma^*(\ell)
  \bigr)$, so it follows that $\ell$ can be expressed as a linear combination of
  the evaluations at points in $Z$.  As in proof of Corollary~4.3 in
  \cite{Blekherman}, we deduce that the set $Z$ contains at most one pair of
  complex zeroes.  Because $\deg(X) \geq 3$, we conclude the set $Z$ must
  contain at least one real zero which produces the required contradiction.
\end{proof}

\section{Equality of the Fundamental Cones}
\label{sec:equality}

\noindent
This section focuses on sufficient conditions for the equality of the
sums-of-squares cone $\Sos_X$ and the nonnegative cone $\Pos_X$.  We complete
the proof of our main theorem, by showing that $\Sos_X$ equals $\Pos_X$ whenever
the quadratic deficiency vanishes.  Combining our main theorem with the
celebrated classification for varieties of minimal degree (e.g.\ Theorem~1 in
\cite{EisenbudHarris}), we describe in detail the varieties for which equality
holds.  Using the Veronese map, we also generalize the main theorem to
nonnegative forms of higher degree.

Our first proposition provides the second implication needed for the proof of
Theorem~\ref{thm:main}.

\begin{proposition}
  \label{pro:if}
  If $\varepsilon(X) = 0$, then we have $\Sos_X = \Pos_X$.
\end{proposition}

\begin{proof}
  It suffices to prove that $\Pos_X^* = \Sos_X^*$.  Given the descriptions for
  $\Pos_X^*$ and $\Sos_X^*$ in Lemma~\ref{lem:cones}, this reduces to showing
  that every extremal ray of $\Sos_X^*$ is generated by evaluation at some point
  $p \in X(\RR)$.  Suppose otherwise and consider an $\ell \in \Sos_X^*$ that
  generates an extremal ray but is not determined by evaluation at a point $p
  \in X(\RR)$.  Lemma~\ref{lem:hsop} establishes that there exists a homogeneous
  system of parameters $g_0, g_1, \dotsc, g_m \in \Ker \bigl( \sigma^*(\ell)
  \bigr)$.  Since $\varepsilon(X) = 0$, Lemma~\ref{lem:deficiency} establishes
  that $X$ is a variety of minimal degree; varieties of minimal degree are
  arithmetically Cohen-Macaulay (e.g.\ see \S4 in \cite{EG}), so $g_0, g_1,
  \dotsc, g_m$ are also a regular sequence.  Let $J$ denote the homogeneous
  ideal in $R$ generated $g_0, g_1, \dotsc, g_m$.  Since we have $\ell(f g_j) =
  0$ for all $f \in R_1$ and all $0 \leq j \leq m$, the linear functional $\ell
  \in R_2^*$ annihilates the subspace $J_2$.  By taking the degree-two graded
  components of the associated Koszul complex and using
  Lemma~\ref{lem:deficiency}, we obtain
  \begin{align*}
    \dim \bigl( \tfrac{R}{J} \bigr)_{\! 2} &= \dim (R_2) - (m+1) \dim (R_1) +
    \tbinom{m+1}{2} \dim (R_0) = \varepsilon(X) = 0 \, ,
  \end{align*}
  whence $R_2 = J_2$.  However, this yields a contradiction because the linear
  functional $\ell \in R_2^*$ is nonzero and does not annihilate all of $R_2$.
  Therefore, every extremal ray of $\Sos_X^*$ is generated by evaluation at some
  point $p \in X(\RR)$ as required.
\end{proof}

\begin{remark}
  In the proof of Proposition~\ref{pro:if}, the hypothesis that $X$ is
  totally-real is not required to establish that $\Pos_X^* = \Sos_X^*$.
\end{remark}

\begin{proof}[Proof of Theorem~\ref{thm:main}]
  If $X$ is not a variety of minimal degree, then we have $\varepsilon(X) > 0$
  and Proposition~\ref{pro:onlyif} establishes that $\Sos_X$ is a proper subset
  of $\Pos_X$.  Conversely, if $X$ is a variety of minimal degree then
  Lemma~\ref{lem:deficiency} establishes that $\varepsilon(X) = 0$ and
  Proposition~\ref{pro:if} states that $\Sos_X = \Pos_X$.
\end{proof}

Beyond the conceptual explanation for the equality $\Pos_X = \Sos_X$,
Theorem~\ref{thm:main} allows us to explicitly exhibit all the varieties that
satisfy this condition.  The classical characterization for varieties of minimal
degree (e.g.\ Theorem~1 in \cite{EisenbudHarris}) states that a variety of
minimal degree is a cone over a smooth variety of minimal degree, and a smooth
variety of minimal degree is either a quadratic hypersurface, the Veronese
surface $\nu_2(\PP^2) \subset \PP^5$, or a rational normal scroll.  Together
with Theorem~\ref{thm:main}, this yields precisely the following three families
in which nonnegativity is equivalent to being a sum of squares.

\begin{example}
  \label{exa:hypersurface}
  Let $X \subset \PP^n$ be a cone over a totally-real irreducible quadric
  hypersurface.  In other words, $R = \RR[x_0,\dotsc, x_n]/I$ where $I$ is the
  principal ideal generated by an indefinite quadratic form.  It follows that
  $\deg(X) = 2 = 1 + \codim(X)$, so Theorem~\ref{thm:main} implies that every
  nonnegative element of $R_2$ is a sum of squares.
\end{example}

\begin{example}
  \label{exa:VeroneseSurface}
  For $n \geq 5$, let $X \subset \PP^n$ be the cone over the Veronese surface
  $\nu_2(\PP^2) \subset \PP^5$.  Given suitable coordinates $x_0, \dotsc, x_n$
  on $\PP^n$, the homogeneous ideal $I$ for $X$ is defined by the $(2 \times
  2)$-minors of the generic symmetric matrix:
  \[
  \begin{bmatrix}
    x_0 & x_1 & x_2 \\
    x_1 & x_3 & x_4 \\
    x_2 & x_4 & x_5 
  \end{bmatrix} \, . 
  \]
  In this case, we have $\deg(X) = 4 = 1 + \codim(X)$, so Theorem~\ref{thm:main}
  implies that every nonnegative element of $R_2 = (\RR[x_0, \dotsc, x_n]/I)_2$
  is a sum of squares. 
\end{example}

\begin{example}
  \label{exa:scroll}
  For $k \geq 0$ and $d_k \geq d_{k-1} \geq \dotsb \geq d_0 \geq 0$ with $d_k >
  0$, set $n := k + d_0 + d_1 + \dotsb + d_k$ and let $X \subset \PP^n$ be the
  associated rational normal scroll; $X$ is the image of the projectivized
  vector bundle $\mathcal{O}_{\PP^1}(d_0) \oplus \mathcal{O}_{\PP^1}(d_1) \oplus
  \dotsb \oplus \mathcal{O}_{\PP^1}(d_k)$ under the complete linear series of
  the tautological line bundle.  In particular, $X$ is the rational normal curve
  of degree $n$ in $\PP^n$ when $k = 0$, and $X$ is $\PP^n$ when $d_{k-1} = 0$
  and $d_k = 1$.  Given suitable coordinates $x_{0,0}, \dotsc, x_{0,d_0},
  x_{1,0}, \dotsc, x_{1,d_1}, \dotsc, x_{k,0}, \dotsc, x_{k,d_k}$ on $\PP^n$,
  the homogeneous ideal $I$ for $X$ is defined by the $(2 \times 2)$-minors of
  the block Hankel matrix:
  \[
  \begin{bmatrix}
    x_{0,0} & \dotsb & x_{0,d_0-1} & x_{1,0} & \dotsb & x_{1,d_1-1} & \dotsb &
    x_{k,0} & \dotsb & x_{k,d_k-1} \\
    x_{0,1} & \dotsb & x_{0,d_0} & x_{1,1} & \dotsb & x_{1,d_1} & \dotsb &
    x_{k,1} & \dotsb & x_{k,d_k}
  \end{bmatrix} \, . 
  \]
  Since we have $\deg(X) = d_0 + d_1 + \dotsb + d_k = n - k = 1 + \codim(X)$,
  Theorem~\ref{thm:main} implies that every nonnegative element of $R_2 =
  (\RR[x^{}_{0,0}, \dotsc, x^{}_{k,d_k}]/I)_2$ is a sum of squares.
\end{example}

The following remark explains why it is sufficient to consider quadratic forms.

\begin{remark}
  \label{rem:higherdegree}
  The union of Theorem~\ref{thm:main} with the classification for varieties of
  minimal degree also allows us to identify when every nonnegative form on $X$
  of degree $2d$ for $d > 1$ is a sum of squares.  Geometrically, this is
  equivalent to recognizing when the $d$-th Veronese embedding of $X \subseteq
  \PP^n$ is a variety of minimal degree.  The degree of every curve on the image
  $\nu_d(X)$ is a multiple of $d$, so $\nu_d(X)$ does not contain any lines.
  Assume that $\nu_d(X)$ is a variety of minimal degree.  It cannot be a cone
  over a smooth variety of minimal degree or a rational normal scroll with $k >
  0$ because these varieties contain lines.  It follows that $\nu_d(X)$ is
  either a rational normal curve or the Veronese surface $\nu_2(\PP^2) \subset
  \PP^5$.  Therefore, every nonnegative form on $X$ of degree $2d$ for $d > 1$
  is a sum of squares if and only if $X \cong \PP^1$ or $X = \PP^2$ and $d =
  2$.  

  As an example, the rational quartic curve in $C \subset \PP^3$ defined by
  $[y_0 : y_1] \mapsto [y_0^4 : y_0^3 y_1 : y_0 y_1^3 : y_1^4]$ is not a variety
  of minimal degree.  However, its image under the second Veronese map $\nu_2(C)
  \subset \PP^8$ is the rational normal curve of degree eight which is a variety
  of minimal degree.  Hence, every nonnegative quartic form on $C$ is a sum of
  squares.
\end{remark}

We conclude this section by viewing our main theorem through the lens of measure
theory.

\begin{remark}
  \label{rem:truncated}
  Fix a positive integer $d$ and let $X$ be a real projective variety with
  homogeneous coordinate ring $R$.  Let $W := S^n \cap \tilde{X}$ be the
  intersection of the affine cone $\tilde{X} \subseteq \AA^{n+1}(\RR)$ of $X$
  with the unit sphere $S^n$.  A measure on $X(\RR)$ corresponds to a measure on
  $W$ which is invariant under the antipodal map.  Any such measure $\mu$
  defines a linear functional $\ell \in R_{2d}^*$ by sending $f \in R_{2d}$ to
  $\int_W f \; d\mu$.  The \define{truncated moment problem} asks for a
  characterization of the $\ell \in R^*_{2d}$ that come from integration with
  respect to a measure on $X$; see Definition~3.1 in \cite{Las}.  Such
  functionals are nonnegative and belong to $\Pos^*_{\nu_d(X)}$.  Moreover,
  every element of $\Pos^*_{\nu_d(X)}$ has this form.  As a result, the
  truncated moment problem on $X$ can be reinterpreted as asking for a
  characterization of the cone $\Pos^*_{\nu_d(X)}$.  If $B_{\ell}$ is the moment
  matrix of $\ell$ (i.e.\ the matrix associated to the quadratic form of $\ell$
  with respect to a monomial basis for $R_d$) then it is a necessary that
  $B_\ell$ be positive semidefinite or equivalently $\ell \in
  \Sos^*_{\nu_d(X)}$.  From this viewpoint, Theorem~\ref{thm:main} classifies
  the varieties $X$ for which the truncated moment problem in degree two is
  equivalent to deciding positive semidefiniteness of the moment matrix.

\end{remark}

\section{The Intrinsic Perspective}
\label{sec:intrinsic}

\noindent
In this section, we shift our perspective from an embedded variety to linear
series on an abstract variety.  This approach gives us greater flexibility which
will be used in applications.  For example, by working with positively
multigraded polynomial rings, we list the cases in which every nonnegative
multihomogeneous polynomial is a sum of squares.

Let $Y$ be an $m$-dimensional totally-real projective variety; it is a
geometrically integral projective scheme over $\Spec(\RR)$ such that the set
$Y(\RR)$ of real points is Zariski dense.  Consider a Cartier divisor $D$ on $Y$
that is locally defined by rational functions with real coefficients, and fix a
nondegenerate basepoint-free linear series $V \subseteq H^0\bigl(Y,
\mathcal{O}_Y(D) \bigr)$.  Since $D$ is defined over $\RR$, we may regard $V$ as
a real vector space.  Let $\sigma \colon \Sym^2 \bigl( H^0 \bigl( Y,
\mathcal{O}_Y(D) \bigr) \bigr) \to H^0 \bigl( Y, \mathcal{O}_Y(2D) \bigr)$
denote the canonical multiplication map and let $2V := \sigma \bigl( \Sym^2(V)
\bigr) \subseteq H^0 \bigl( Y, \mathcal{O}_Y(2D) \bigr)$.  Given a real point $p
\in Y(\RR)$ and a section $s \in H^0 \bigl( Y, \mathcal{O}_Y(2D) \bigr)$, the
\define{sign of $s$ at $p$} is $\sign_{p}(s) := \sign ( \lambda ) \in
\{-1,0,1\}$ where $U \subseteq Y$ is a neighbourhood of the point $p \in Y$ over
which the line bundle $\mathcal{O}_Y(D)$ is trivial, $\varsigma \in H^0 \bigl(
U, \mathcal{O}_{Y}(D) \bigr)$ is a generator of $\mathcal{O}_Y(D)|_{U}$, and
$\lambda \in H^{0}(U,\mathcal{O}_Y)$ is defined by $s|_{U} = \lambda
\varsigma^2$.  The sign of $s$ at $p$ is independent of the choice of $U$ and
$\varsigma$; see \S2.4 in \cite{Scheiderer}.  The section $s$ is
\define{nonnegative} if $\sign_{p}(s) \geq 0$ for all $p \in Y(\RR)$ and we
simply write $s(p) \geq 0$.

The central objects of study, in this intrinsic setting, become
\begin{alignat*}{2}
  \Pos_{Y,V} &:= \{ s \in 2V : \text{$s(p) \geq 0$ for all $p \in Y(\RR)$} \} \,
  , & &\text{and}\\
  \Sigma_{Y,V} &:= \{ s \in 2V : \text{there exist $t_1, t_2, \dotsc, t_k \in V$
    such that $s = \sigma(t_1^2) + \sigma(t_2^2) + \dotsb + \sigma(t_k^2)$ } \}
  \, . &
\end{alignat*}
We again have $\Sos_{Y,V} \subseteq \Pos_{Y,V}$.  To describe the properties of
these subsets, let $n$ be the projective dimension of $|V|$, let $\varphi \colon
Y \to \PP^n$ be the associated morphism, and let $X := \varphi(Y)$.  The
linear series $|V|$ is nondegenerate if and only if $X \subseteq \PP^n$ is
nondegenerate.  The kernel of the composition of the canonical homomorphisms of
graded rings $\RR[x_0, \dotsc, x_n] \cong \Sym(V) \to \Sym \bigl( H^0(Y,
\mathcal{O}_Y(D) \bigr)$ and $\Sym \bigl( H^0(Y, \mathcal{O}_Y(D) \bigr) \to
\bigoplus\nolimits_{j \in \NN} H^0 \big( Y, \mathcal{O}_Y(jD) \bigr)$ is the
unique saturated ideal $I$ vanishing on $X$.  It follows that the homogeneous
coordinate ring of $X$ is $R = \Sym(V)/I$, and the induced inclusion of graded
rings is $\varphi^{\sharp} \colon R \to \bigoplus\nolimits_{j \in \NN} H^0
\big( Y, \mathcal{O}_Y(jD) \bigr)$.

The next proposition shows that these collections of $\Pos_{Y,V}$ and
$\Sos_{Y,V}$ are closely related to the cones $\Pos_X$ and $\Sos_X$, and
provides an alternative version of Theorem~\ref{thm:main}.

\begin{theorem}
  \label{thm:linearseries}
  We have $\varphi^{\sharp}( \Sos_X) = \Sos_{Y,V}$.  If $\varphi
  \bigl(Y(\RR)\bigr)$ is dense in the strong topology on $X(\RR)$, then we also
  have $\varphi^{\sharp}(\Pos_X) = \Pos_{Y,V}$, and $\Pos_{Y,V} = \Sos_{Y,V}$ if
  and only if $X$ is a variety of minimal degree.
\end{theorem}

\begin{proof} 
  By construction, we have $\varphi^{\sharp}(R_1) = V$ and
  $\varphi^{\sharp}(R_2) = 2V$, which establishes the first assertion.  Since
  $\varphi$ sends a real point to a real point, we have $\Pos_{Y,V} \subseteq
  \varphi^{\sharp}(\Pos_X)$.  Conversely, each real point in $X$ lies in the
  closure of the image of a real point in $Y$ by assumption, so we have
  $\varphi^{\sharp}(\Pos_X) \subseteq \Pos_{Y,V}$.  Combining the first two
  parts with Theorem~\ref{thm:main} yields the third part.
\end{proof}

\begin{remark}
  \label{rem:odd}
  When the map $\varphi$ has finite fibers of odd length, the condition on
  $\varphi$ in Theorem~\ref{thm:linearseries} is automatically satisfied.  In
  particular, the hypothesis holds when $\varphi$ is an embedding.  Indeed,
  complex conjugation fixes the fiber over a real point.  Since the fibers have
  odd length, conjugation must fix at least one point in each fiber over a real
  point, so $\varphi$ maps $Y(\RR)$ surjectively onto $X(\RR)$.
\end{remark}

Without placing some restrictions on the map $\varphi$, the theorem is false.

\begin{example}
  Consider the linear series $V = \langle x_0^2, x_1^2, \dotsc, x_n^2 \rangle
  \subseteq H^0 \bigl( \PP^n, \mathcal{O}_{\PP^n}(2) \bigr)$.  The corresponding
  morphism $\varphi \colon \PP^n \to \PP^n$ is not surjective on real points.
  In this case, $(\varphi^{\sharp})^{-1}(\Pos_{Y,V})$ consists of all quadratic
  forms that are nonnegative on the closed nonnegative orthant in $\RR^{n+1}$
  (i.e.\ the copositive forms) and this collection is strictly larger than the
  cone of all nonnegative quadratic forms; see \S3.6.1 in \cite{BPT}.
\end{example}

The following explains why we can restrict to linear series for which $2V =
H^0\bigl( Y, \mathcal{O}_Y(2D) \bigr)$.

\begin{observation}
  \label{obs:2normal}
  If $2V \neq H^0 \bigl( Y, \mathcal{O}_Y(2D) \bigr)$, then we claim that there
  is a nonnegative section in $H^0 \bigl( Y, \mathcal{O}_Y(2D) \bigr)$ that is
  not a sum of squares.  Since the linear series $V$ is basepoint-free, there
  exists $t_0, t_1, \dotsc, t_n \in V$ with no common zeroes, so $\sigma(t_0^2)
  + \sigma(t_1^2) + \dotsb + \sigma(t_n^2) \in H^0 \bigl(Y, \mathcal{O}_Y(2D)
  \bigr)$ is strictly positive on $Y(\RR)$.  Our assumption on $2V$ implies that
  there is a section $s \in H^0 \bigl( Y, \mathcal{O}_{Y}(2D) \bigr) \setminus
  2V$.  It follows that the section $\sigma(t_0^2) + \sigma(t_1^2) + \dotsb +
  \sigma(t_n^2) - \delta s \in H^0 \bigl( Y, \mathcal{O}_{Y}(2D) \bigr)$ cannot
  be a sum of squares for all $\delta \in \RR$.  On the other hand, this section
  is nonnegative for all sufficiently small $\delta > 0$, because $Y(\RR)$ is
  compact set and, for any section $s$, we have
  \[
  \overline{\{ p \in X(\RR) : \sign_p(s) < 0 \}} \subseteq \{ p \in X(\RR) : \sign_p(s)
  \leq 0 \} \, . 
  \]
\end{observation}

To illustrate the power of Theorem~\ref{thm:linearseries}, we capture all of the
previously known situations in which nonnegativity is equivalent to being a sum
of squares.

\begin{example} 
  \label{exa:Hilbert}
  For $n \geq 0$ and $d \geq 1$, consider $Y = \PP^n$ and $V = H^{0} \bigl(
  \PP^n, \mathcal{O}_{\PP^n}(d) \bigr)$.  The corresponding map $\varphi$
  is the Veronese embedding, so Theorem~\ref{thm:linearseries} implies that
  every nonnegative homogeneous polynomial of degree $2d$ is a sum of squares
  (i.e.\ $\Pos_{Y,V} = \Sos_{Y,V}$) if and only if $X = \varphi(\PP^n)$ is
  a variety of minimal degree.  Moreover, we have $\deg(X) = d^n =
  \binom{n+d}{n} -n = 1 + \codim(X)$ in only three cases:
  \begin{itemize}
  \item $n = 1$: all nonnegative binary forms are sums of squares, and
    $X$ is a rational normal curve;
  \item $d = 1$: all nonnegative quadratic forms are sums of squares, and $X =
    \PP^n$;
  \item $d = 2$ and $n = 2$: all nonnegative ternary quartics are sums of
    squares, and $X$ is the Veronese surface.
  \end{itemize}
  In particular, we recover Hilbert's famous characterization of when every
  nonnegative homogeneous polynomial is a sum of squares; see \cite{Hilbert} or
  \cite{BPT}*{\S3.1.2}.  Even better, we provide a new geometric interpretation
  for the exceptional case of ternary quartics. 
\end{example}

\begin{example} 
  \label{exa:CLR}
  For $k \geq 2$, $n_i \geq 1$, and $d_i \geq 1$ where $1 \leq i \leq k$,
  consider $Y = \PP^{n_1} \times \PP^{n_2} \times \dotsb \times \PP^{n_k}$ and
  the linear series $V = H^{0} \bigl( Y, \mathcal{O}_{\PP^{n_1}}(d_1) \boxtimes
  \mathcal{O}_{\PP^{n_2}}(d_2) \boxtimes \dotsb \boxtimes
  \mathcal{O}_{\PP^{n_k}}(d_k) \bigr)$.  The corresponding map $\varphi$ is the
  Segre-Veronese embedding, so Theorem~\ref{thm:linearseries} implies that every
  nonnegative multihomogeneous polynomial of degree $(2d_1, \dotsc, 2d_k)$ is a
  sum of squares (i.e.\ $\Pos_{Y,V} = \Sos_{Y,V}$) if and only if $X =
  \varphi(Y)$ is a variety of minimal degree.  Moreover, we have
  \begin{align*}
    \deg(X) &= d_1^{n_1} d_2^{n_2} \dotsb d_k^{n_k} \tbinom{n_1 + n_2 + \dotsb +
      n_k}{n_1, n_2, \dotsc, n_k} = d_1^{n_1} d_2^{n_2} \dotsb d_k^{n_k}
    \tfrac{(n_1 + n_2 + \dotsb +
      n_k)!}{n_1! n_2! \dotsb n_k!} \\
    &= \tbinom{n_1+d_1}{n_1}\tbinom{n_2+d_2}{n_2} \dotsb \tbinom{n_k+d_k}{n_k} -
    n_1^{} - n_2^{} - \dotsb - n_k^{} = 1 + \codim(X)
  \end{align*}
  in precisely two cases:
  \begin{itemize}
  \item $k = 2$, $n_1 = 1$, and $d_2 = 1$,
  \item $k = 2$, $n_2 = 1$, and $d_1 = 1$.
  \end{itemize}
  By symmetry, both cases assert that all nonnegative biforms that are quadratic
  in one set of variables and binary in the other set of variables are sums of
  squares, and $X$ is a rational normal scroll associated to a vector bundle of
  the form $\bigoplus_{j} \mathcal{O}_{\PP^1}(1)$.  In other words, we recover
  and provide a new geometric interpretation for Theorem~8.4 in \cite{CLR}.
\end{example}

Since two of the three families of varieties of minimal degree are toric
varieties, the intrinsic descriptions can be expressed in terms of a polynomial
ring with an appropriate grading.

\begin{example}
  \label{exa:VeroneseCone}
  For $n \geq 5$, consider the cone $Y \subset \PP^{n}$ over the Veronese
  surface $\PP^2 \hookrightarrow \PP^5$ and the complete linear series $V =
  H^{0} \bigl( Y, \mathcal{O}_Y(H) \bigr)$ where $H$ is a hyperplane divisor
  (cf.\ Example~\ref{exa:VeroneseSurface}).  Hence, $Y$ is a simplicial normal
  toric variety with class group $\ZZ^1$ and the Cox homogeneous coordinate ring
  is $S := \RR[y_0, \dotsc, y_{n-3}]$ where $\deg(y_i) = 1$ for $0 \leq i \leq
  2$ and $\deg(y_j) = 2$ for $3 \leq j \leq n-3$.  Since $\operatorname{Pic}(X)$
  has index two within the class group, it follows that $V = H^{0} \bigl( Y,
  \mathcal{O}_Y(H) \bigr) \cong S_2$.  The image of $Y$ is a variety of minimal
  degree, so Theorem~\ref{thm:linearseries} implies that every nonnegative
  element in $S_4$ is a sum of squares.  An element of $S_4$ is a linear
  combination of the $15$ monomials of the form $y_0^4, y_0^3 y_1^{}, \dotsc,
  y_2^4$, the $6n-30$ monomials of the form $y_0^2 y_j^{}, y_0^{} y_1^{} y_j^{},
  \dotsc, y_2^2 y_j^{}$ where $3 \leq j \leq n-3$, and the $\binom{n-4}{2}$
  monomials of the form $y_3^2, y_3^2 y_4^{}, \dotsc, y_{n-3}^2$; the vector
  space $S_4$ has dimension $\tfrac{1}{2}n^2 +\tfrac{3}{2}n -5$.  Contrary to
  the sentence preceding Theorem~8.4 in \cite{CLR}, this inserts the exceptional
  case of ternary quartics from Example~\ref{exa:Hilbert} into an infinite
  family.
\end{example}

\begin{example}
  \label{exa:smoothScroll}
  For integers $k > 0$ and $d_k \geq d_{k-1} \geq \dotsb \geq d_0 > 0$, consider
  the projectivized vector bundle $Y = \PP \bigl( \mathcal{O}_{\PP^1}(d_0)
  \oplus \mathcal{O}_{\PP^1}(d_1) \oplus \dotsb \oplus \mathcal{O}_{\PP^1}(d_k)
  \bigr)$ and the complete linear series $V = H^{0}\bigl( Y, \mathcal{O}_Y(1)
  \bigr)$ (cf.\ Example~\ref{exa:scroll}).  Hence, $Y$ is a $(k+1)$-dimensional
  smooth toric variety with class group $\ZZ^2 = \operatorname{Pic}(X)$; see
  pages~6--7 in \cite{EisenbudHarris}.  By choosing a suitable basis for the
  class group, the Cox homogeneous coordinate ring is $S := \RR[y_0, \dotsc,
  y_{k+2}]$ where the degree of $y_j$ in $\ZZ^2$ is given by the $j$-th column
  of the matrix
  \[
  \left[ \begin{matrix}
      1 & 1 & 0 &  d_0-d_1 & d_0 - d_2 & \dotsb & d_0 - d_k  \\
      0 & 0 & 1 & 1 & 1 & \dotsb  & 1
    \end{matrix} \right] \, .
  \]  
  It follows that $V = H^{0}\bigl( Y, \mathcal{O}_Y(1) \bigr) \cong S_{(1,1)}$.
  Since the image of $Y$ is a variety of minimal degree,
  Theorem~\ref{thm:linearseries} implies that every nonnegative element in
  $S_{(2,2)}$ is a sum of squares.  An element of $S_{(2,2)}$ is a linear
  combination of monomials that are quadratic in the variables $y_2, y_3,
  \dotsc, y_{k+2}$; the vector space $S_{(2,2)}$ has dimension $(3-2d_0)
  \binom{k+2}{2} + (k+2)(d_0 + d_1 + \dotsb + d_k)$.  The special case $d_0 =
  d_1 = \dotsb = d_k = 1$ retrieves Example~\ref{exa:CLR}.
\end{example}

\begin{remark}
  \label{rem:nonsmooth}
  Example~\ref{exa:smoothScroll} excludes two types of rational normal scrolls:
  a cone over a rational normal curve (i.e.\ $k = 0$ or $d_{k-1} = 0$) which has
  class group isomorphic to $\ZZ^1$ and a cone over a smooth rational normal
  scroll (i.e.\ $d_{k-1} \neq 0$ and $d_0 = 0$) which has class group isomorphic
  to $\ZZ^2$.  The minor modifications to Example~\ref{exa:smoothScroll}
  required for both types are left to the interested reader.
\end{remark}

\begin{remark}
  \label{rem:matrix}
  The multihomogeneous forms in Example~\ref{exa:smoothScroll} also have a
  useful matrix interpretation.  By viewing $f \in S_{(2,2)}$ as a quadratic
  form in the variables $y_2, y_3, \dotsc, y_{k+2}$, we obtain a symmetric
  matrix $F$ with homogeneous entries in $\RR[y_0,y_1]$.  Lemma~3.78 in
  \cite{BPT} basically shows that $F$ is pointwise positive semidefinite if and
  only if $f$ is nonnegative and $F = G_1^{\textsf{T}} G_1 + G_2^{\textsf{T}}
  G_2 + \dotsb + G_k^{\textsf{T}} G_k$ for some matrices $G_1, G_2, \dotsc, G_k$
  with entries in $\RR[y_0,y_1]$ if and only if $f$ is a sum of squares.  Hence,
  the fact that every nonnegative element in $S_{(2,2)}$ is a sum of squares
  becomes a slight strengthening of Theorem~3.80 in \cite{BPT} in which each
  entry is homogeneous (although not necessarily of the same degree).
\end{remark}

\section{Nonnegative Sparse Polynomials}
\label{sec:sparse}

\noindent
This section examines certain sparse Laurent polynomials---those Laurent
polynomials in which the exponent vector of each monomial appearing with a
nonzero coefficient lies in a fixed lattice polytope.  We characterize the
Newton polytopes $Q$ such that every nonnegative polynomial with support
contained in $2Q$ is a sum of squares.

Let $M$ be an $m$-dimensional affine lattice, let $M_{\RR} := M \otimes_{\ZZ}
\RR$ be the associated real vector space, and let $T := \Spec(\RR[M])$ be the
corresponding split real torus.  Choosing an isomorphism $M \cong \ZZ^m$
identifies the group ring $\RR[M]$ with the Laurent polynomial ring
$\RR[z_1^{\pm 1}, \dotsc, z_m^{\pm 1}]$.  Given $f = \sum_{\mathbf{u} \in M}
c_{\mathbf{u}} \, z^{\mathbf{u}} \in \RR[M]$, its \define{Newton polytope} is
$\New(f) := \conv \{ \mathbf{u} \in M : c_{\mathbf{u}} \neq 0 \} \subset M_\RR$.
The Laurent polynomial $f$ is \define{nonnegative}, denoted by $f \geq 0$, if
the evaluation of $f$ at every point in $T(\RR)$ is nonnegative.  Fix an
$m$-dimensional lattice polytope $Q$ in $M_\RR$.  For $k \in \NN$, write $kQ$ is
the $k$\nobreakdash-fold Minkowski sum of $Q$.  The lattice polytope $Q$ is
\define{$k$-normal} if, for each $\mathbf{u} \in (kQ) \cap M$, there exist
$\mathbf{v}_1, \mathbf{v}_2, \dotsc, \mathbf{v}_k \in Q \cap M$ such that
$\mathbf{u} = \mathbf{v}_1 + \mathbf{v}_2 + \dotsb + \mathbf{v}_k$, cf.\
Definition~2.2.9 in \cite{CLS}.  Following \S3 in \cite{Stan}, the
\define{$\hh^*$-polynomial} of $Q$ is $\hh^*_0(Q) + \hh^*_1(Q) \, t + \dotsb +
\hh^*_m(Q) \, t^m := (1-t)^{m+1} \sum_{k \geq 0} \bigl| (kQ) \cap M \bigr| \,
t^k$.

The central objects of study, in this polyhedral setting, are
\begin{xalignat*}{2}
  \Pos_Q &:= \{ f \in \RR[M] : \text{$\New(f) \subseteq 2 Q$ and $f \geq 0$}
  \} & & \text{and} \\
  \Sos_Q &:= \left\{ f \in \RR[M] : 
    \begin{array}{l}
      \text{there exists $g_1, g_2, \dotsc, g_k \in
        \RR[M]$ such that $\New(g_j) \subseteq Q$} \\
      \text{for all $1 \leq j \leq k$ and $f = g_1^2 + g_2^2 + \dotsb + g_k^2$}
    \end{array}
  \right\} \, .
\end{xalignat*}
Once again, we have $\Sos_Q \subseteq \Pos_Q$.  To describe the properties of
these subsets, let $X \subseteq \PP^n$ be the embedded projective toric variety
determined by the lattice polytope $Q$.  More explicitly, the number of lattice
points in $Q$ is $n+1 = |Q \cap M|$, the polyhedral affine monoid associated to
$Q$ is $C(Q) := \NN \cdot \{ (q,1) : q \in Q \cap M \} \subset M \oplus \ZZ$,
and the toric variety is $X = \Proj \bigl( \RR[C(Q)] \bigr) \subseteq \PP^n$;
cf.\ \S2.3 in \cite{CLS}.  The lattice points in $Q$ also yield the canonical
inclusion map $\eta \colon T \to X$.

\begin{example}
  \label{exa:degree1}
  If $Q$ is an $(m-2)$-fold pyramid over the simplex $\conv\{ (0,0), (2,0),
  (0,2)\} \subset \RR^2$, then the embedded projective toric variety $X$ is the
  cone over the Veronese surface defined in Example~\ref{exa:VeroneseSurface}.
  Likewise, if $Q$ is the Cayley polytope of the line segments $[0,d_0],
  [0,d_1], \dotsc, [0,d_k]$ (see Definition~2.1 in \cite{BN}), then the embedded
  projective toric variety $X$ is the rational normal scroll defined in
  Example~\ref{exa:scroll} .
\end{example}

To establish that $2$-normality is a necessary condition for $\Pos_Q = \Sos_Q$,
we have a better version of Observation~\ref{obs:2normal} which provides an
explicit bound on the coefficient $\delta$.

\begin{lemma}
  \label{lem:2normal}
  If $Q \subset M_\RR$ is a lattice polytope that is not $2$-normal, then
  $\Sos_Q$ is a proper subset of $\Pos_Q$.
\end{lemma}

\begin{proof}
  Since $Q$ is not $2$-normal, there exists a lattice point $\mathbf{u} \in 2Q
  \cap M$ that cannot be written as a sum of lattice points in $Q \cap M$.  If
  $\mathbf{v}_1, \mathbf{v}_2, \dotsc, \mathbf{v}_k$ denote the vertices of $Q$,
  then $\mathbf{u}$ is a convex rational linear combination of $2 \mathbf{v}_1,
  2 \mathbf{v}_2, \dotsc, 2 \mathbf{v}_k$ which are the vertices of $2Q$.  By
  clearing the denominators, we obtain $(r_1 + r_2 + \dotsb + r_k) \mathbf{u} =
  2 r_1 \mathbf{v}_1 + 2 r_2 \mathbf{v}_2 + \dotsb + 2 r_k \mathbf{v}_k$ where
  $r_1, r_2, \dotsc, r_k \in \NN$ and $r_1 + r_2 + \dotsb + r_k > 0$.  Consider
  the Laurent polynomial 
  \[
  f := r_1 z^{2 \mathbf{v}_1} + r_2 z^{2 \mathbf{v}_2} + \dotsb + r_k z^{2
    \mathbf{v}_k} - (r_1 + r_2 + \dotsb + r_k) z^{\mathbf{u}} \, .
  \]
  Clearly, $\New(f) \subseteq 2Q$, and our choice of $\mathbf{u}$ guarantees
  that $f$ is not a sum of squares.  On the other hand, the inequality of
  weighted arithmetic and geometric means shows that $f$ is nonnegative.
  Therefore, we have $f \in \Pos_Q \setminus \Sos_Q$.
\end{proof}

The following result is a strengthening of Theorem~\ref{thm:linearseries} for
projective toric varieties, because the condition on real points is now both
necessary and sufficient.

\begin{theorem}
  \label{thm:sparse}
  We have $\Pos_Q = \Sigma_Q$ if and only if $\hh^*_2(Q) = 0$ and $\eta \bigl(
  T(\RR) \bigr)$ is dense in the strong topology on $X(\RR)$.
\end{theorem}

\begin{proof}
  We first verify that $Q$ is $2$-normal.  If $\Pos_Q = \Sigma_Q$, then
  Lemma~\ref{lem:2normal} shows that $Q$ is $2$-normal.  Assuming that
  $\hh^*_2(Q) = 0$, we confirm that $Q$ is $2$-normal by induction on the
  dimension $m$.  Since every lattice polytope of dimension at most $2$ is
  normal (i.e.\ $k$-normal for all $k$), the base case for the induction holds.
  If $m \geq 3$, then our assumption together with inequality (4) in \cite{S2}
  proves that $\hh^*_m(Q) = 0$.  Similarly, inequality (6) in \cite{S2} (with $i
  = 1$) shows that $\hh^*_{m-1}(Q) = 0$.  Hence, Ehrhart--Macdonald reciprocity
  (e.g.\ Theorem~4.4 in \cite{BR}) establishes that neither $Q$ nor $2Q$ have
  any interior lattice points.  It follows that every lattice point $\mathbf{u}
  \in (2Q) \cap M$ is contained in a face of $2Q$.  Since every facet of $2Q$
  equals $2F$ for some face $F$ of $Q$ and the monotonicity of
  $\hh^*$-polynomials (i.e.\ Theorem~3.3 in \cite{Stan}) ensures that
  $\hh^*_2(F) \leq \hh^*_2(Q) = 0$, the induction hypothesis shows that $F$ is
  $2$-normal.  In particular, we have $\mathbf{u} = \mathbf{v}_1 + \mathbf{v}_2$
  for some $\mathbf{v}_1, \mathbf{v}_2 \in F \cap M \subset Q \cap M$ and we
  conclude that $Q$ is also $2$-normal.

  The $2$-normality of $Q$ ensures that $R_2 = \RR[C(Q)]_2 \cong \RR \cdot \{
  (2Q) \cap M \}$ and, by definition, we have $\hh^*_1(Q) = n+1 = | Q \cap M | = \dim
  \RR[C(Q)]_1$, which together imply that $\Pos_Q = \Pos_X$ and $\Sos_Q =
  \Sos_X$.  Since $\hh^*_0(Q) = 1 = \dim \RR[C(Q)]_0$, Lemma~\ref{lem:deficiency}
  establishes that $\hh^*_2(Q) = \varepsilon(X)$ and we have $\hh^*_2(Q) = 0$ if and only
  if $X$ is a variety of minimal degree.  If $\eta \bigl( T(\RR) \bigr)$ is
  dense in the strong topology on $X(\RR)$, then Theorem~\ref{thm:main} proves
  that $\Pos_Q = \Sos_Q$ if and only if $\hh^*_2(Q) = \varepsilon(X) = 0$.  Thus, it
  remains to show that $\Pos_Q = \Sos_Q$ implies that $\eta \bigl( T(\RR)
  \bigr)$ is dense in the strong topology on $X(\RR)$.

  Assume $\Pos_Q = \Sos_Q$ and suppose that $\eta\bigl( T(\RR) \bigr)$ is not
  dense in the strong topology on $X(\RR)$.  By translating $Q$ in $M_{\RR}$ if
  necessary, we may assume that $Q$ contains the origin and this lattice point
  corresponds to the $0$-th coordinate of the map $\eta \colon T \to X \subseteq
  \PP^n$.  Let $U_0 \cong \AA^n$ denote the distinguished open subset of $\PP^n$
  determined by the vanishing of the $0$-th coordinate and set $W := X \cap U_0
  \subset \AA^n$.  Since $\eta \bigl( T(\RR) \bigr) \subseteq W$, our
  supposition implies that $\eta \bigl( T(\RR) \bigr)$ is not dense in the
  strong topology on $W(\RR)$.  As a consequence, there exists a point $p \in
  W(\RR)$ and a real number $\delta > 0$ such that the open ball $B_{\delta}(p)$
  of radius $\delta$ centered at $p$ is completely contained in $W(\RR)
  \setminus \overline{\eta \bigl( T(\RR) \bigr)}$.  Choose coordinates $x_0,
  x_1, \dotsc, x_n$ on $\PP^n$ with $p = [1 : p_1 : p_2 : \dotsb : p_n] \in
  \PP^n(\RR)$.  Consider the polynomial $\hat{f} := (x_1-p_1 x_0)^2 + (x_2-p_2
  x_0)^2 + \dotsb + (x_n-p_n x_0)^2 - \delta x_0^2 \in \RR[x_0, \dotsc, x_n]$
  and the corresponding Laurent polynomial $f = \eta^{\sharp}(\hat{f}) \in
  \RR[M]$ where $\eta^{\sharp} \colon \RR[x_0, \dotsc, x_n] \to \RR[M]$ is the
  canonical ring homomorphism associated to $\eta$.  By construction, we have
  $\New(f) \subseteq 2Q$ and $f$ is nonnegative on $T(\RR)$, so $f \in \Pos_Q$.
  The assumption $\Pos_Q = \Sos_Q$ guarantees that there exists $g_1, g_2,
  \dotsc, g_k \in \RR[M]$ such that $f = g_1^2 + g_2^2 + \dotsb + g_k^2$.  It
  follows that $\New(g_j) \subseteq \tfrac{1}{2} \New(f) = Q$, so there are
  linear forms $\hat{g}_j \in \RR[x_0, \dotsc, x_n]$ satisfying $g_j =
  \eta^{\sharp}(\hat{g})$ for $1 \leq j \leq k$.  Since $\eta^{\sharp}$ is
  injective, we obtain $\hat{f} = \hat{g}_1^2 + \hat{g}_2^2 + \dotsb +
  \hat{g}_k^2$.  However, this is impossible because $\hat{f}(p) = - \delta <
  0$.  Therefore, we conclude that $\eta\bigl( T(\RR) \bigr)$ is dense in the
  strong topology on $X(\RR)$.
\end{proof}

The ensuing propositions, which practically classify the lattice polytopes $Q$
with $\hh^*_2(Q) = 0$, increase the utility of Theorem~\ref{thm:sparse}.  They
also advance the general program of classifying polytopes based on their
$\hh^*$-polynomials.

\begin{proposition}
  \label{pro:h^*_2=0}
  Let $Q \subset M_{\RR}$ be an $m$-dimensional lattice polytope.  We have
  $\hh^*_2(Q) = 0$ if and only if $Q$ is $2$-normal and $Q$ is the affine
  $\ZZ$-linear image, surjective on integral points, of a polytope $Q'$ where
  $Q' \subset M'_{\RR}$ is either the $(m-2)$-fold pyramid over $\conv\{ (0,0),
  (2,0), (0,2)\} \subset \RR^2$ or the Cayley polytope of $m$ line segments.
\end{proposition}

\begin{proof}
  The first paragraph in the proof of Theorem~\ref{thm:sparse} shows that $Q$ is
  $2$-normal whenever $\hh^*_2(Q) = 0$, and the second paragraph shows that the
  $2$-normality of $Q$ implies that $0 = h_2^* = \varepsilon(X)$ and $X$ is a
  variety of minimal degree.  Since $X$ is a toric variety, the classification
  for varieties of minimal degree (e.g.\ Theorem~1 in \cite{EisenbudHarris})
  establishes that $X$ is either a cone over the Veronese surface or a rational
  normal scroll.  It follows from Example~\ref{exa:degree1} that $X$ is
  projectively equivalent to the embedded toric variety $X'$ determined by a
  polytope $Q'$ where $Q'$ is either an $(m-2)$-fold pyramid over $\conv\{
  (0,0), (2,0), (0,2)\} \subset \RR^2$ or the Cayley polytope of $m$ line
  segments.  The $\RR$-algebras $\RR[C(Q)]$ and $\RR[C(Q')]$ are isomorphic, so
  Theorem~2.1 in \cite{Gub} implies that the affine monoids $C(Q)$ and $C(Q')$
  are also isomorphic.  This isomorphism extends to a $\ZZ$-linear homomorphism
  $\beta \colon M' \oplus \ZZ \to M \oplus \ZZ$, because $C(Q')$ contains a
  lattice basis, and $\beta$ is injective, because $Q$ is full-dimensional.
  Restricting to the affine slice at height $1$, we obtain the affine map
  $\alpha \colon M' \to M$ such that $\alpha(Q') = Q$.  Since $\beta$, and hence
  $\alpha$, sends the generators of $C(Q')$ to the generators of $C(Q)$, every
  lattice point in $Q$ is the image of a lattice point in $Q'$.
\end{proof}

\begin{corollary}
  \label{cor:det}
  Let $Q' \subset M'_{\RR}$ be either the $(m-2)$-fold pyramid over $\conv\{
  (0,0), (2,0), (0,2)\} \subset \RR^2$ or the Cayley polytope of $m$ line
  segments, and let $\alpha \colon M' \to M$ be an affine map.  If $Q :=
  \alpha(Q')$ and the determinant of the linear component of $\alpha$ is a
  nonzero odd integer, then we have $\Pos_Q = \Sos_Q$.
\end{corollary}

\begin{proof}
  Propositon~\ref{pro:h^*_2=0} implies that $\hh^*_2(Q) = 0$, so it is enough to
  prove, by Theorem~\ref{thm:sparse}, that $\eta\bigl( T(\RR) \bigr)$ is dense
  in the strong topology on $X(\RR)$.  The embedded projective toric variety $X
  \subseteq \PP^n$ determined $Q$ is a compactification of the dense algebraic
  torus $T'' := X \cap \{ x_0 x_1 \dotsb x_n \neq 0 \}$, so it suffices to show
  that the induced map $\eta'' \colon T(\RR) \to T''(\RR)$ obtained from $\eta$
  is surjective.  If $M''$ denotes the sublattice generated by the lattice
  points in $Q$, then induced map $\eta''$ corresponds to an injective ring
  homomorphism from $\RR[M''] \to \RR[M]$.  Since $\RR[M'']$ is the image of map
  $\RR[M'] \to \RR[M]$ defined by the linear component of $\alpha$, it follows
  that $\eta''$ is a finite morphism with degree equal to the determinant of the
  linear component.  As in Remark~\ref{rem:odd}, $\eta''$ is surjective when the
  degree is odd.
\end{proof}

To refine our classification, we need an auxiliary invariant: the
\define{degree} of $Q$ is the smallest nonnegative integer $j$ such that, for $1
\leq k \leq m-j$, $k Q$ contains no interior lattice point.

\begin{remark}
  One can directly verify that a pyramid over $\operatorname{conv}\{ (0,0),
  (2,0), (0,2)\} \subset \RR^2$ or a Cayley polytope of line segments has degree
  one.
\end{remark}

With a few small adjustments to the proof of Proposition~\ref{pro:h^*_2=0}, we
obtain the following.

\begin{proposition}
  \label{pro:degree1}
  For an $m$-dimensional lattice polytope $Q \subset M_\RR$, the following are
  equivalent:
  \begin{enumerate}
  \item[(a)] $Q$ is normal and $\hh^*_2(Q) = 0$,
  \item[(b)] $Q$ is a polytope of degree one,
  \item[(c)] we have $\hh^*_2(Q) = \hh^*_3(Q) = \dotsb = \hh^*_m(Q) = 0$.
  \end{enumerate}
\end{proposition}

\begin{proof} (a) $\Longrightarrow$ (b): Since $\hh^*_2(Q) = 0$, the proof of
  Proposition~\ref{pro:h^*_2=0} provides $\ZZ$-linear homomorphism $\beta \colon
  M' \oplus \ZZ \to M \oplus \ZZ$.  By changing bases on the source and target,
  we can assume that $\beta$ is represented by a diagonal matrix (i.e.\ its
  Smith normal form) which sends a lattice basis in $C(Q')$ to certain
  multiplies in $C(Q)$.  Since $Q$ is normal, the monoid $C(Q)$ also contains a
  lattice basis.  It follows that $\beta$ is a lattice isomorphism.  By
  restricting to the affine slice at height $1$, we conclude that $Q$ and $Q'$
  are affinely isomorphic.
  
  (b) $\Longrightarrow$ (c): As in the proof of Proposition~\ref{pro:h^*_2=0},
  this follows immediately from Ehrhart--Macdonald reciprocity (e.g.\
  Theorem~4.4 in \cite{BR}).

  (c) $\Longrightarrow$ (a): We need to show that $Q$ is $k$-normal for all $k >
  1$.  Since every $m$-dimensional polytope is $k$-normal for all $k \geq m-1$
  (e.g.\ Theorem~2.2.12 in \cite{CLS}), we may assume $k < m-1$.  For $2 \leq k
  \leq m-j$, one can adapt the arguments from the first paragraph in the proof
  of Proposition~\ref{pro:h^*_2=0} to show $Q$ is $k$-normal.
\end{proof}

\begin{remark}
  \label{rem:BN}
  By combining Proposition~\ref{pro:h^*_2=0} and the proof of
  Propositon~\ref{pro:degree1}, we obtain a new interpretation and a new proof
  for the main theorem in \cite{BN}.  Specifically, Theorem~2.5 in \cite{BN}
  characterizes the $m$-dimensional lattice polytopes of degree one as either an
  $(m-2)$-fold pyramid over the simplex $\conv\{ (0,0), (2,0), (0,2)\} \subset
  \RR^2$ or the Cayley polytope of $m$ line segments.
\end{remark}

We end with a family of non-normal polytopes $Q$ for which $\hh^*_2(Q) = 0$.  By
examining the proof of Proposition~\ref{pro:h^*_2=0}, we see that smallest such
example must have dimension at least $5$.

\begin{example}
  Let $m \geq 5$ be an odd integer and fix $k \in \NN$.  If $\mathbf{e}_1,
  \dotsc, \mathbf{e}_m$ denotes the standard basis for $\ZZ^m$, then consider
  the simplex 
  \[
  Q := \conv\{ \mathbf{0}, \mathbf{e}_1, \dotsc, \mathbf{e}_m, \mathbf{e}_1 +
  \dotsb + \mathbf{e}_{(m-1)/2} + k \, \mathbf{e}_{(m+1)/2} + \dotsb + k \,
  \mathbf{e}_{m-1} + (k+1) \, \mathbf{e}_m \} \, .
  \]
  Section~1 in \cite{H} shows that the $\hh^*$-polynomial for $Q$ is $1 + k \,
  t^{(m+1)/2}$, so $\hh^*_2(Q) = 0$.  When $k$ is even, Corollary~\ref{cor:det}
  implies that $\Pos_Q = \Sigma_Q$.  When $k$ is odd, $\eta\bigl(T(\RR) \bigr)$
  is not dense in the strong topology on $X(\RR)$, so Theorem~\ref{thm:sparse}
  implies that $\Pos_Q \neq \Sigma_Q$.
\end{example}


\raggedright


\begin{bibdiv}
\begin{biblist}

\bib{BN}{article}{
  author={Batyrev, Victor},
  author={Nill, Benjamin},
  title={Multiples of lattice polytopes without interior lattice points},
  journal={Mosc. Math. J.},
  volume={7},
  date={2007},
  number={2},
  pages={195--207, 349},
}

\bib{BR}{book}{
  author={Beck, Matthias},
  author={Robins, Sinai},
  title={Computing the continuous discretely},
  series={Undergraduate Texts in Mathematics},
  publisher={Springer},
  place={New York},
  date={2007},
  pages={xviii+226},
}

\bib{Becker}{article}{
  label={Bec},
  author={Becker, Eberhard},
  title={Valuations and real places in the theory of formally real fields},
  conference={
    title={Real algebraic geometry and quadratic forms},
    address={Rennes},
    date={1981},
  },
  book={
    series={Lecture Notes in Math.},
    volume={959},
    publisher={Springer},
    place={Berlin},
  },
  date={1982},
  pages={1--40},
}	

\bib{Blekherman}{article}{
  label={Ble},
  author={Blekherman, Grigoriy},
  title={\href{http://dx.doi.org/10.1090/S0894-0347-2012-00733-4}%
    {Nonnegative polynomials and sums of squares}},
  journal={J. Amer. Math. Soc.},
  volume={25},
  date={2012},
  number={3},
  pages={617--635},
  issn={0894-0347},
}

\bib{BPT}{book}{
  author={Blekherman, Grigoriy},
  author={Parrilo, Pablo A.},
  author={Thomas, Rekha R.},
  title={Semidefinite optimization and convex algebraic geometry},
  series={MOS-SIAM Series on Optimization},
  volume={13},
  publisher={Society for Industrial and Applied Mathematics (SIAM)},
  place={Philadelphia, PA},
  date={2013}
}

\bib{BIK}{article}{
  author={Blekherman, Grigoriy},
  author={Iliman, Sadik},
  author={Kubitzker, Martina},
  title={Dimensional differences between faces of the cones of nonnegative
    polynomials and sums of squares},
  note={available at \href{http://arxiv.org/abs/1305.0642}%
  {\texttt{arXiv:1305.0642 [math.AG]}}}
}

\bib{BS}{article}{
  author={Brodmann, Markus},
  author={Schenzel, Peter},
  title={\href{http://dx.doi.org/10.1090/S1056-3911-06-00442-5}%
    {Arithmetic properties of projective varieties of almost minimal degree}},
  journal={J. Algebraic Geom.},
  volume={16},
  date={2007},
  number={2},
  pages={347--400},
}

\bib{CLS}{book}{
  author={Cox, David A.},
  author={Little, John B.},
  author={Schenck, Henry K.},
  title={Toric varieties},
  series={Graduate Studies in Mathematics},
  volume={124},
  publisher={American Mathematical Society},
  place={Providence, RI},
  date={2011},
  pages={xxiv+841},
}

\bib{CLR}{article}{
  author={Choi, Man Duen},
  author={Lam, Tsit Yuen},
  author={Reznick, Bruce},
  title={\href{http://dx.doi.org/10.1007/BF01215051}%
    {Real zeros of positive semidefinite forms. I}},
  journal={Math. Z.},
  volume={171},
  date={1980},
  number={1},
  pages={1--26},
}


\bib{EG}{article}{
  author={Eisenbud, David},
  author={Goto, Shiro},
  title={\href{http://dx.doi.org/10.1016/0021-8693(84)90092-9}%
    {Linear free resolutions and minimal multiplicity}},
  journal={J. Algebra},
  volume={88},
  date={1984},
  number={1},
  pages={89--133},
}

\bib{EisenbudHarris}{article}{
  author={Eisenbud, David},
  author={Harris, Joe},
  title={On varieties of minimal degree (a centennial account)},
  conference={
    title={Algebraic geometry, Bowdoin, 1985},
    address={Brunswick, Maine},
    date={1985},
  },
  book={
    series={Proc. Sympos. Pure Math.},
    volume={46},
    publisher={Amer. Math. Soc.},
    place={Providence, RI},
  },
  date={1987},
  pages={3--13}
}

\bib{Gub}{article}{
  label={Gub},
  author={Gubeladze, Joseph},
  title={\href{http://dx.doi.org/10.1016/S0022-4049(97)00063-7}%
    {The isomorphism problem for commutative monoid rings}},
  journal={J. Pure Appl. Algebra},
  volume={129},
  date={1998},
  number={1},
  pages={35--65},
}

\bib{Harris}{book}{
  label={Har},
  author={Harris, Joe},
  title={Algebraic geometry. A first course},
  series={Graduate Texts in Mathematics},
  volume={133},
  publisher={Springer-Verlag},
  place={New York},
  date={1992}
}

\bib{H}{article}{
  label={Hig},
  author={Higashitani, Akihiro},
  title={\href{http://dx.doi.org/10.1007/s00454-011-9390-4}%
      {Counterexamples of the conjecture on roots of Ehrhart polynomials}},
  journal={Discrete Comput. Geom.},
  volume={47},
  date={2012},
  number={3},
  pages={618--623},
}

\bib{Hilbert}{article}{
  label={Hil},
  author={Hilbert, David},
  title={\href{http://dx.doi.org/10.1007/BF01443605}%
    {\"{U}ber die Darstellung definiter Formen als Summe von Formenquadraten}},
  journal={Math. Ann.},
  volume={32},
  date={1888},
  number={3},
  pages={342--350},
}

\bib{J}{book}{
  label={Jou},
  author={Jouanolou, Jean-Pierre},
  title={Th\'eor\`emes de Bertini et applications},
  publisher={Birkh\"auser},
  series={Progress in Mathematics},
  volume={42},
  place={Basel},
  date={1983}
}

\bib{Las}{book}{
  label={Las},
  author={Lasserre, Jean Bernard},
  title={Moments, positive polynomials and their applications},
  series={Imperial College Press Optimization Series},
  volume={1},
  publisher={Imperial College Press},
  place={London},
  date={2010},
  pages={xxii+361},
}

\bib{Lau}{article}{
  label={Lau},
  author={Laurent, Monique},
  title={\href{http://dx.doi.org/10.1007/978-0-387-09686-5_7}%
    {Sums of squares, moment matrices and optimization over polynomials}},
  conference={
    title={Emerging applications of algebraic geometry},
  },
  book={
    series={IMA Vol. Math. Appl.},
    volume={149},
    publisher={Springer},
    place={New York},
  },
  date={2009},
  pages={157--270},
}

\bib{Lvovsky}{article}{
  label={L{\cprime}v},
  author={L{\cprime}vovsky, S.},
  title={\href{http://dx.doi.org/10.1007/BF01445273}%
    {On inflection points, monomial curves, and hypersurfaces containing
      projective curves}},
  journal={Math. Ann.},
  volume={306},
  date={1996},
  number={4},
  pages={719--735},
}

\bib{RamanaGoldman}{article}{
  author={Ramana, Motakuri},
  author={Goldman, A. J.},
  title={\href{http://dx.doi.org/10.1007/BF01100204}%
    {Some geometric results in semidefinite programming}},
  journal={J. Global Optim.},
  volume={7},
  date={1995},
  number={1},
  pages={33--50},
}

\bib{Reznick}{article}{
  label={Rez},
  author={Reznick, Bruce},
  title={\href{http://dx.doi.org/10.1090/conm/253/03936}%
    {Some concrete aspects of Hilbert's 17th Problem}},
  conference={
    title={Real algebraic geometry and ordered structures (Baton Rouge,
      LA, 1996)},
  },
  book={
    series={Contemp. Math.},
    volume={253},
    publisher={Amer. Math. Soc.},
    place={Providence, RI},
  },
  date={2000},
  pages={251--272},
}

\bib{Scheiderer}{article}{
  label={Sch},
  author={Scheiderer, Claus},
  title={\href{http://dx.doi.org/10.1007/s00229-011-0484-3}%
    {A Positivstellensatz for projective real varieties}},
  journal={Manuscripta Math.},
  volume={138},
  date={2012},
  number={1-2},
  pages={73--88},
}

\bib{Stan}{article}{
  label={St1},
  author={Stanley, Richard P.},
  title={\href{http://dx.doi.org/10.1006/eujc.1993.1028}%
    {A monotonicity property of $h$-vectors and $h^*$-vectors}},
  journal={European J. Combin.},
  volume={14},
  date={1993},
  number={3},
  pages={251--258},
  issn={0195-6698},
}

\bib{S2}{article}{
  label={Sta},
  author={Stapledon, Alan},
  title={\href{http://dx.doi.org/10.1090/S0002-9947-09-04776-X}%
    {Inequalities and Ehrhart $\delta$-vectors}},
  journal={Trans. Amer. Math. Soc.},
  volume={361},
  date={2009},
  number={10},
  pages={5615--5626},
}

\bib{Zak}{article}{
  label={Zak},
  author={Zak, Fedor L.},
  title={\href{http://dx.doi.org/10.1007/s002080050271}%
    {Projective invariants of quadratic embeddings}},
  journal={Math. Ann.},
  volume={313},
  date={1999},
  number={3},
  pages={507--545},
}
		
\end{biblist}
\end{bibdiv}
\end{document}